%% file: TCofnPointsOnATree.tex
\documentclass[11pt]{amsart}   
\usepackage[all]{xy}
\usepackage{graphicx,enumerate}
\usepackage [english]{babel}
\usepackage{pinlabel}

\title{Topological complexity of $n$ points on a tree}
\author{Steven Scheirer}
\address{Department of Mathematics, Lehigh University\\Bethlehem, PA  18015, United States}
\email{sts413@lehigh.edu}
\urladdr{https://www.lehigh.edu/~sts413}

\newcommand{\Z}{\mathbb{Z}}
\newcommand{\Q}{\mathbb{Q}}
\newtheorem{thm}{Theorem}[section]
\newtheorem{prop}[thm]{Proposition}
\newtheorem{lemma}[thm]{Lemma}
\newtheorem{defn}[thm]{Definition}

\newtheorem{cor}[thm]{Corollary}

\def\qua{\hskip.75em\relax}
\def\co{\colon\thinspace}

\input{Labeling/LabeledFigures}

\begin{document}

\begin{abstract} 
The topological complexity of a path-connected space $X,$ denoted by $TC(X),$ can be thought of as the minimum number of continuous rules needed to describe how to move from one point in $X$ to another.  The space $X$ is often interpreted as a configuration space in some real-life context.  Here, we consider the case where $X$ is the space of configurations of $n$ points on a tree $\Gamma.$  We will be interested in two such configuration spaces.  In the, first, denoted by $C^n(\Gamma),$ the points are distinguishable, while in the second, $UC^n(\Gamma),$ the points are indistinguishable.   We determine $TC(UC^n(\Gamma))$ for any tree $\Gamma$ and many values of $n,$ and consequently determine $TC(C^n(\Gamma))$ for the same values of $n$ (provided the configuration spaces are path-connected).
\end{abstract}

\maketitle

\section{Introduction}\label{sec:Intro}
For any topological space $X,$ let $P(X)$ be the space of continuous paths $\sigma\co[0,1]\to X$ equipped with the compact-open topology.  There is a fibration $p\co P(X)\to X\times X$ which sends a path $\sigma$ to its endpoints: $p(\sigma)=(\sigma(0),\sigma(1)).$  When studying the problem of motion planning within a topological space, one often wishes to find sections of this fibration.  That is, one wishes to find functions $s\co X\times X\to P(X)$ such that $p\circ s$ is the identity.  Such a function takes a pair of points as input and produces a path between those points, hence the relation to motion planning.  The continuity of a section $s$ at a point $(x,y)\in X\times X$ means that if $(x',y')$ is ``close" to $(x,y),$ then the path $s(x',y')$ is ``close" to the path $s(x,y).$  Unfortunately it is a rarity that such a function can be continuous over all of $X\times X$ (in fact, such a continuous section exists if and only if the space $X$ is contractible, see \cite{FarberTC}).  This leads to the definition of topological complexity introduced by Farber in \cite{FarberTC}: 

\begin{defn}\label{def:TC}
For any path-connected space $X,$ the topological complexity of $X$, denoted by $TC(X),$ is the smallest integer $k\ge1$ such that there is a cover of $X\times X$ by open sets $U_1,U_2,\dots,U_k$ and continuous sections $s_i\co U_i\to P(X).$  If there is no such $k,$ set $TC(X)=\infty.$
\end{defn}
Such a collection of sets $U_i$ and sections $s_i$ is called a \emph{motion planning algorithm.}  Thus, $TC(X)$ is in some sense the smallest number of continuous rules required to describe how to move between any two points in $X.$   The space $X$ is often viewed as the space of configurations of some real-world system.  One example is when $X$ is the space of configurations of $n$ robots which move around a factory along a system of one-dimensional tracks.  Such a system of tracks can be interpreted as a graph $\Gamma$ (a one-dimensional CW complex).  There are two types of these configuration spaces:  in the first, denoted by $C^n(\Gamma),$ the robots are distinguishable, and in the second, denoted by $UC^n(\Gamma),$ the robots are indistinguishable.  In other words, in $C^n(\Gamma),$ both the points of $\Gamma$ occupied by robots, and the specific robots which occupy those points are of interest, while in $UC^n(\Gamma),$ it is only required that each specified point in $\Gamma$ is occupied by \emph{some} robot, but it is irrelevant which specific robot occupies each point.  There are different real-world situations in which one configuration space is preferable over the other, depending on whether or not the robots are to perform different tasks.  Our main goal is to study the topological complexity of the configuration spaces $C^n(\Gamma)$ and $UC^n(\Gamma)$ when $\Gamma$ is a tree (a tree is a connected graph which has no cycles).  The topological complexity of these configuration spaces is related to the number of vertices of degree greater than 2 (the degree of a vertex is defined in Section \ref{sec:graphconfig}).  These vertices are called \emph{essential vertices}, and $m(\Gamma)$ is the number of essential vertices in $\Gamma.$  Here, an \emph{arc} in $\Gamma$ is a subspace homeomorphic to a non-trivial closed interval.  We will be interested in certain collections of arcs in $\Gamma$ which are called \emph{allowable} and will be defined in Definition \ref{defn:AllowableOrientations}.  In Section \ref{sec:treeMP}, we establish the following:

\begin{thm}\label{thm:Main} Let $\Gamma$ be a tree with $m:=m(\Gamma)\ge 1.$
\begin{enumerate}
\item\label{nLarge} Let $k$ be the smallest integer such that there is a collection of oriented arcs $\{A_i\}_{i=1}^k$ which is allowable for the collection of all vertices of degree 3 in $\Gamma.$ If there are no vertices of degree 3, let $k=0.$  Let $n\ge 2m+k$ be an integer.  Then, $TC(C^n(\Gamma))=TC(UC^n(\Gamma))=2m+1.$
\item\label{nNotLarge} Let $n=2q+\epsilon<2m,$ with $\epsilon\in\{0,1\}$ and $q\ge1,$ let $r$ be the number of vertices of degree greater than 3, and let $s$ be the number of vertices of degree 3. Suppose one of the these three cases hold:
\begin{enumerate}
\item\label{nSmall} $s\ge 2(q-r).$ 
\item \label{nMedium}\begin{enumerate}
\item\label{even}$s<2(q-r),\ \epsilon=0,$ and there is some $k\ge1$ such that there exists a collection of oriented arcs $\{A_i\}_{i=1}^k$ with the following properties:
\begin{enumerate}
\item The endpoints of each $A_l$ are (distinct) essential vertices, neither of which is an endpoint of any other $A_{l'},$
\item  There are $r'\le r$ vertices of degree greater than 3 which are not the endpoints of any $A_l,$
\item  There is a collection $\mathcal{V}$ of degree-3 vertices, with $|\mathcal{V}|\ge q-r'-k$ such that $\{A_i\}_{i=1}^k$ is allowable for $\mathcal{V}.$  
\end{enumerate}
\item\label{odd} $s<2(q-r),\ \epsilon=1,$ and there is an arc $A_0$ whose endpoints have no restrictions and whose interior includes a collection $\mathcal{W}'$ of $s'\le q$ distinct vertices of degree 3, and if $s'<q-r,$ there are arcs $A_1,\dots,A_k,$ as above whose endpoints are also not vertices in $\mathcal{W}',$ and there is another collection of degree-3 vertices, $\mathcal{W}$, such that $\mathcal{W}\cap\mathcal{W}'=\emptyset,\ |\mathcal{W}|\ge q-r'-k-s'$ and $\{A_i\}_{i=1}^k$ is allowable for $\mathcal{W},$ where $r'$ is as above.
\end{enumerate} 
\end{enumerate}
Then $TC(C^n(\Gamma))=TC(UC^n(\Gamma))=2q+1.$ 
\end{enumerate}
\end{thm}

In \cite{FarberConfigSpacesMotionPlanning}, Farber proves a similar statement for the spaces $C^n(\Gamma),$ showing that 
\[
TC(C^n(\Gamma))=2m(\Gamma)+1,\quad\text{ for $m(\Gamma)\ge1$ and $n\ge 2m(\Gamma),$}
\]
 with an additional assumption that $\Gamma$ is not homeomorphic to the letter $Y$ if $n=2.$  His proof uses methods which differ from the ones used here, and only address the spaces $C^n(\Gamma)$ for $n$ large.  Theorem \ref{thm:Main} in some sense improves this, since it addresses both the spaces $C^n(\Gamma)$ and $UC^n(\Gamma),$ and includes some values of $n$ less than $2m(\Gamma).$  Furthermore, Corollary \ref{cor:NoDegreeThree} determines the topological complexity of both configuration spaces of a tree with no vertices of degree 3 for all values of $n\ge1,$ provided the configuration spaces are connected.  On the other hand, with $k\ge1$ as above, Farber's results determine $TC(C^n(\Gamma))$ for $n=2m,2m+1,\dots,2m+k-1,$ while our result does not.  We discuss in Proposition \ref{prop:nobetter} the extent to which the results of Theorem \ref{thm:Main} are the best one can achieve with the methods used here.

\section{Configuration spaces of points on graphs}\label{sec:graphconfig}
Consider a graph $\Gamma,$ where, as above, a graph is a 1-dimensional CW complex.  The zero-dimensional cells of $\Gamma$ are the vertices, and the closures of the 1-dimensional cells are the edges.  The degree of a vertex $v$ is the number of edges which have $v$ as exactly one of their endpoints plus twice the number of edges which have $v$ as both endpoints.  We will deal exclusively with finite graphs, so that the number of vertices and edges is finite.  An essential vertex is a vertex of degree equal to or greater than 3, and $m(\Gamma)$ is the number of essential vertices in $\Gamma.$  Let $C^n(\Gamma)$ be the space of $n$-tuples of distinct points in $\Gamma.$  That is, 
\[
C^n(\Gamma)=\overbrace{\Gamma\times\dots\times\Gamma}^{n\ times}-\Delta,
\]
where $\Delta=\{(x_1,\dots,x_n)\in \Gamma\times\cdots\times\Gamma:x_i=x_j \text{ for some }i\ne j\}.$  The space $C^n(\Gamma)$ will be called the \emph{topological configuration space of $n$ ordered points on $\Gamma.$}  Similarly, let 
\[
D^n(\Gamma)=\overbrace{\Gamma\times\dots\times\Gamma}^{n\ times}-\widetilde{\Delta},
\]
where $\widetilde{\Delta}$ consists of all products of cells in $\Gamma\times\cdots\times\Gamma$ whose closures intersect $\Delta.$  Thus, $D^n(\Gamma)$ consists of all products of cells $c_1\times\dots\times c_n$ such that $\overline{c}_i\cap\overline{c}_j=\emptyset$ whenever $i\ne j.$  In what follows, the word ``cell" will always refer to the closure of a cell.  A point in $D^n(\Gamma)$ is then an ordered $n$-tuple of points $(x_1,\dots,x_n)$ in $\Gamma$ such that there is at least a full open edge between two distinct points $x_i$ and $x_j.$  The space $D^n(\Gamma)$ will be called the \emph{discrete configuration space of $n$ ordered points on $\Gamma.$}  

There is a free action of the symmetric group $S_n$ on $C^n(\Gamma)$ and $D^n(\Gamma)$ which permutes the coordinates.  The quotients of these two spaces under this action are denoted by $UC^n(\Gamma)$ and $UD^n(\Gamma),$ and are called the \emph{unordered} topological and discrete configuration spaces.  Given a point $\mathbf{y}=(y_1,\dots,y_n)\in C^n(\Gamma),$ we may find neighborhoods $U_i$ in $\Gamma$ which contain $y_i$ and satisfy $U_i\cap U_j=\emptyset$ for $i\ne j.$  Then, $U=U_1\times\cdots\times U_n$ is a neighborhood of $\mathbf{y}$ in $C^n(\Gamma)$ and for each $\alpha\in S_n,$ we have $\alpha(U)=U_{\alpha(1)}\times\cdots\times U_{\alpha(n)}.$  So, if $\alpha\ne\alpha'\in S_n,$ then there is some $i$ such that $\alpha(i)\ne\alpha'(i),$ and then $U_{\alpha(i)}\cap U_{\alpha'(i)}=\emptyset,$ so $\alpha(U)\cap\alpha'(U)=\emptyset.$  This implies that the quotient map $C^n(\Gamma)\to UC^n(\Gamma)$ is a covering space projection (see \cite[Proposition 1.40]{Hatcher}).  If $\mathbf{y}$ is a point in $D^n(\Gamma),$ by letting $U'=U\cap D^n(\Gamma),$ we see that the same is true of the quotient map $D^n(\Gamma)\to UD^n(\Gamma).$  The topological and discrete spaces are related by the following:

\begin{thm}{\rm \cite{AbramsThesis},\cite{KimKoPark}}\qua \label{thm:defretract}
Let $\Gamma$ be a graph with at least $n$ vertices.  Suppose 
\begin{enumerate}
\item \label{defretractppty1} each path between distinct vertices of degree not equal to 2 in $\Gamma$ contains at least $n-1$ edges, and 
\item each loop at a vertex in $\Gamma$ which is not homotopic to a constant map contains at least $n+1$ edges.
\end{enumerate}
Then, $C^n(\Gamma)$ and $UC^n(\Gamma)$ deformation retract onto $D^n(\Gamma)$ and $UD^n(\Gamma),$ respectively.
\end{thm}

In \cite{AbramsThesis}, Abrams proves a slightly weaker version of Theorem \ref{thm:defretract}, assuming that each path as in item (\ref{defretractppty1}) contains $n+1$ edges, and conjectures the stronger version given here.  Kim, Ko, and Park prove this conjecture in \cite{KimKoPark}.  A graph which satisfies these conditions is called \emph{sufficiently subdivided for $n$.}  Any graph can be made sufficiently subdivided for any $n$ by adding enough degree-2 vertices; this has no effect on the topology of the graph or either of the topological configuration spaces.  The following relates the connectivity of a graph $\Gamma$ and the connectivity of the topological configuration spaces:

\begin{thm}{\rm\cite{AbramsThesis}}\qua \label{thm:connectivity}
Suppose $\Gamma$ is a graph with at least one edge or at least $n+1$ vertices.  Then,
\begin{enumerate}
\item $C^n(\Gamma)$ is path-connected if and only if $\Gamma$ is connected and either
\begin{enumerate}
\item $n=1,$ 
\item $n=2$ and $\Gamma$ is not homeomorphic to a closed interval, or
\item $n\ge 3$ and $\Gamma$ is not homeomorphic to a closed interval or a circle.
\end{enumerate}
\item $UC^n(\Gamma)$ is path-connected if and only if $\Gamma$ is connected.
\end{enumerate}
\end{thm}
It follows that if $\Gamma$ is sufficiently subdivided, then in Theorem \ref{thm:connectivity}, $C^n(\Gamma)$ and $UC^n(\Gamma)$ can be replaced with $D^n(\Gamma)$ and $UD^n(\Gamma).$   In fact, Abrams also shows that if $\Gamma$ has at least $n+1$ vertices, then $UD^n(\Gamma)$ is connected if and only if $\Gamma$ is connected, regardless of subdivision.  If $\Gamma$ has exactly $n$ vertices, one can easily find examples in which $\Gamma$ is disconnected, but $UD^n(\Gamma)$ is a single point (so $UD^n(\Gamma)$ is connected). There are less trivial examples in \cite{AbramsThesis} in which $\Gamma$ is connected, while $D^n(\Gamma)$ is disconnected.  

Daniel Farley and Lucas Sabalka have studied extensively the homotopy and homology groups and the cohomology rings of the spaces $UD^n(\Gamma)$ for a sufficiently subdivided tree $\Gamma$ using Forman's discrete Morse theory \cite{Forman}.   Recall a tree is a simply connected graph.  We summarize some of their results which will be relevant here.  From here on, $\Gamma$ is a tree which is sufficiently subdivided for $n$.

First, an ordering on the vertices is constructed as follows.  Embed the tree $\Gamma$ in the plane, and let $\ast$ be a vertex of degree 1. Assign $\ast$ the number 0, travel away from $\ast,$ and number the remaining vertices in order (starting with 1) when they are first encountered.  Whenever an essential vertex is encountered, take the leftmost edge, and turn around when a vertex of degree 1 is encountered.  For each edge $e,$ let $\iota(e)$ and $\tau(e)$ be the two endpoints of $e,$ with $\iota(e)<\tau(e).$   There is also a notion of \emph{directions} from a vertex $v\ne\ast$ of degree $d.$  These directions are a numbering of the edges incident to $v$ from 0 to $d-1,$ in increasing order clockwise around the vertex, with 0 being the direction on the geodesic segment from $v$ to $\ast.$ An example is given in Figure \ref{fig:VertexOrdering}.  The 0-direction at each essential vertex is marked with an arrow;  to avoid clutter, the other directions are not marked on the graph.  Note this graph is only sufficiently subdivided for $n=2,$ since there is only one edge along the geodesic from vertex 12 to vertex 15.

\VertexOrdering

Recall an arc in $\Gamma$ is a subspace homeomorphic to a non-trivial closed interval.  Farley and Sabalka's notion of directions enables us to define the notion of an allowable collection of oriented arcs.  Given a finite collection of oriented arcs $\{A_i\}_{i=1}^k$ in $\Gamma,$ and a vertex $v$ in $\Gamma$ of degree $d,$ we will  define integers $\eta_0(v),\eta_1(v),\dots,\eta_{d-1}(v)$ as follows.  First, for $i=1,2,\dots,k$ and $j=0,1,\dots,d-1,$ if $v$ falls on the arc $A_i$ and $A_i$ intersects the interior of the edge $e_j$ incident to $v$ in direction $j,$ let $\eta_{j,i}(v)=1$ if $A_i$ is oriented towards $v$ on $e_j$ and $\eta_{j,i}(v)=-1$ if $A_i$ is oriented away from $v$ on $e_j.$  If $v$ does not fall on $A_i$ or if $A_i$ does not intersect the interior of $e_j,$ let $\eta_{j,i}(v)=0.$  Then, let 
\[
\eta_j(v)=\sum_{i=1}^k\eta_{j,i}(v).
\]
This leads to the following definition:
\begin{defn}\label{defn:AllowableOrientations}
Suppose $\mathcal{A}=\{A_i\}_{i=1}^k$ is a collection of oriented arcs in $\Gamma$ and $\mathcal{V}$ is a collection of vertices of $\Gamma.$  The collection $\mathcal{A}$ is said to be allowable for $\mathcal{V}$ if every vertex $v\in\mathcal{V}$ of degree $d$ has the property that $v$ is not an endpoint of any $A_i,$ and at least one of $\eta_0(v),\ \eta_1(v),\dots,\eta_{d-1}(v)$ is non-zero.
\end{defn}
Intuitively, a collection of oriented arcs is allowable for a given collection $\mathcal{V}$ of vertices if at each vertex in $\mathcal{V},$ there is at least one direction in which the orientations of the arcs don't ``cancel out," and no vertex in $\mathcal{V}$ is an endpoint of any arc.  Also, if $\mathcal{V}=\emptyset,$ then any collection of arcs is allowable for $\mathcal{V}.$  We will be most interested in the case in which $\mathcal{V}$ is a subset of the vertices of degree 3.  Figure \ref{fig:AllowableOrientations} shows an example of two collections of oriented arcs in a graph $\Gamma.$  The arcs are shown with a dashed line and orientations indicated with arrows, and have been moved away from $\Gamma$ so that they are distinguishable in the figure.  If $\mathcal{V}$ is the collection of all essential vertices of $\Gamma$ (which are all of degree 3), and $\mathcal{W}$ is the collection of all essential vertices of $\Gamma$ except the vertex labeled $v,$ then both collections of arcs are allowable for $\mathcal{W},$ but the collection of arcs on the left is not allowable for $\mathcal{V},$ while the collection on the right is.

\AllowableOrientations

Note that given a collection of oriented arcs $\mathcal{A}=\{A_i\}_{i=1}^k$ which is allowable for $\mathcal{V}$ and a vertex $v\in\mathcal{V}$ of degree $d$, if $j$ is a direction satisfying $\eta_j(v)\ne0,$ then there must be some other direction $j'\ne j$ with $\eta_{j'}(v)\ne0.$  Indeed if $v$ falls on an arc $A_i,$ then $v$ must fall on the interior of $A_i,$ so that $\eta_{j_1,i}(v)=1$ for some $j_1,$ and $\eta_{j_2,i}(v)=-1$ for some $j_2\ne j_1.$  Since $\Gamma$ is a tree, for all other directions $j,$ we have $\eta_{j,i}(v)=0,$ so $A_i$ contributes 0 to the sum 
\[
\sum_{j=0}^{d-1}\sum_{i=1}^k\eta_{j,i}(v)=\sum_{j=0}^{d-1}\eta_j(v).
\]
On the other hand, if $v$ does not fall on $A_i,$ then $\eta_{j,i}(v)=0$ for all $j,$ so again, $A_i$ contributes 0 to the sum above.  So, we have $\eta_0(v)+\eta_1(v)+\cdots+\eta_{d-1}(v)=0,$ showing that there cannot be exactly one value $j$ for which $\eta_j(v)\ne0.$  Furthermore, for the same reason, if $\eta_j(v)\ne0,$ then there must be some direction $j'\ne j$ such that $\eta_j(v)$ and $\eta_{j'}(v)$ have opposite signs.  

A cell $c$ of $UD^n(\Gamma)$ can be described as a collection of vertices and edges, $c=\{c_1,\dots,c_n\},$ where each $c_i$ is a vertex or an edge, and $c_i\cap c_j=\emptyset$ for $i\ne j.$ Note the order in which the vertices and edges appear in $c$ does not matter.  The dimension of the cell $c$ is the number of edges in $c.$  Consider a cell $c=\{v,c_2,\dots,c_n\}$ containing some vertex $v\ne\ast.$  If $e_v$ is the unique edge in $\Gamma$ which has $\tau(e_v)=v,$ and $\{e_v,c_2,\dots,c_n\}$ is a valid cell in $UD^n(\Gamma),$ then $v$ is said to be \emph{unblocked} in $c.$  Otherwise, $v$ is \emph{blocked} in $c.$  The vertex $\ast$ is also said to be blocked in any cell containing it.  In other words, $v$ is unblocked in $c$ if and only if $v\ne \ast$ and $v$ can be replaced with the edge which contains $v$ and is on the geodesic segment from $v$ to $\ast.$  

Now, consider a cell $c=\{e,c_2,\dots,c_n\}$ which contains some edge $e.$  If there is some vertex $v$ in $c$ which has the property that $\iota(e_v)=\iota(e)$ and $\iota(e)<v<\tau(e),$ then $e$ is said to be \emph{order-disrespecting} in $c.$  Otherwise, the edge $e$ is \emph{order-respecting} in $c.$  Figure \ref{fig:Cells} gives examples of 3 different cells in $UD^2(\Gamma)$ with $\Gamma$ as in Figure \ref{fig:VertexOrdering}.  The vertices and edges which are to be included in a cell $c$ are labeled; all unlabeled vertices and edges are not included in the cell.   In the left cell, the vertex 3 is blocked, since $e_3$ intersects the edge labeled $e.$  Also, the edge $e$ is order-disrespecting since $\iota(e_3)=2=\iota(e)$ and $\iota(e)=2<3<\tau(e)=6.$  In the middle cell, the vertex 10 is blocked, but 9 is unblocked.  In the cell on the right, the vertex 16 is blocked, but $f$ is order-respecting, since although $\iota(e_{16})=\iota(f),$ we have $\tau(f)=13<16.$  

\Cells

Farley and Sabalka construct a discrete vector field $W$ and prove the following classification of the critical, collapsible, and redundant cells in $UD^n(\Gamma)$ with respect to $W.$  The terms \emph{critical, collapsible}, and \emph{redundant} come from discrete Morse theory, but their definitions will not be needed here.
\begin{thm}{\rm\cite{FarleySabalka}}\qua \label{thm:critcells}
A cell $c$ of $ UD^n(\Gamma)$ is critical if and only if each vertex in $c$ is blocked and each edge in $c$ is order-disrespecting.  A cell $c$ is collapsible if and only if it contains some order-respecting edge $e$ with the property that any unblocked vertex $v$ in $c$ satisfies $v>\tau(e).$  All other cells are redundant.  
\end{thm}
 
Since we will be focused primarily on critical cells, we describe a procedure to construct a critical $k$-cell in $UD^n(\Gamma).$ First, notice that if $c$ is any $k$-cell, then $c$ must consist of $k$ edges and $n-k$ vertices.  If $c$ is to be critical, each edge must be order-disrespecting, and each vertex must be blocked.  Consider an edge $e$ in $c.$  If $e$ is to be order-disrespecting, then $\iota(e)$ must be an essential vertex, or else there would be no possible vertex $v$ that could satisfy $\iota(e_v)=\iota(e)$ other than $v=\tau(e),$ but $\tau(e)$ cannot be included in a cell which contains $e$ (by this, we mean that $\tau(e)$ cannot appear as a vertex in the list of vertices and edges which define $c$).  Furthermore, the direction $d$ from $\iota(e)$ on which $e$ falls must be at least 2, and the direction $d'$ from $\iota(e)$ on which $v$ falls must satisfy $0<d'<d.$  Note also that if $v$ causes $e$ to be order-disrespecting, then $v$ is automatically blocked.  This also implies that if there is a critical $k$-cell in $UD^n(\Gamma),$ then $k\le m(\Gamma)$ and $n\ge 2k.$  The remaining $n-2k$ vertices of $c$ can be easily chosen so that they are blocked in $c.$ The cell on the left in Figure \ref{fig:Cells} is a critical 1-cell in $UD^2(\Gamma).$  Figure \ref{fig:CriticalCells} gives an example of a critical 3-cell in $UD^8(\Gamma).$  Note that more vertices of degree 2 must be added to the tree $\Gamma$ above so that it is sufficiently subdivided for $n=8.$  In this example and what follows, we make no indication of the total number of vertices in a sufficiently subdivided tree.

\CriticalCells

The discussion above is similar to the proof of the following theorem:

\begin{thm}{\rm\cite{FarleySabalka}}\qua \label{thm:kcells}
Let $\Gamma$ be a tree, and let $k=\min\left\{\left\lfloor\frac{n}{2}\right\rfloor,m(\Gamma)\right\}.$  If $c$ is any critical cell in $UD^n(\Gamma),$ then $\mathrm{dim}(c)\le k.$  Furthermore, $UD^n(\Gamma)$ deformation retracts onto the space $(UD^n(\Gamma))'_k,$ which consists of the $k$-skeleton of $UD^n(\Gamma)$ with the redundant $k$-cells removed.
\end{thm}
The second statement of this theorem follows from results from discrete Morse theory.  Now, before discussing the cohomology ring $H^*(UD^n(\Gamma);\Z),$ we describe the equivalence relation on cells given in \cite{FarleySabalkaCohomology}.  Given two cells $c$ and $c'$ of $UD^n(\Gamma),$  define $``\sim"$ by $c\sim c'$ if and only if $c$ and $c'$ share the same edges (so in particular $c$ and $c'$ are of the same dimension), and if $E$ is the set of edges in $c$ (and in $c'$), and $\mathcal{C}=\Gamma-E,$ then for every connected component $C$ of $\mathcal{C},$ the number of vertices of $c$ in $C$ equals the number of vertices of $c'$ in $C.$  Here and in what follows, we use $E$ to denote both the set of edges and the union of the edges in $E.$  Context should make the desired interpretation of $E$ clear.  Let $[c]$ denote the equivalence class of $c.$  Now, given two equivalence classes $[c]$ and $[d],$ write $[d]\le[c]$ if  there are representatives $c\in[c]$ and $d\in[d]$ such that $d$ is obtained from $c$ by removing some (possibly zero) edges of $c$ and replacing each of these edges with one of its endpoints.  Farley and Sabalka show the following:

\begin{lemma}{\rm\cite{FarleySabalkaCohomology}}\qua \label{lemma:partialorder}
The relation $``\le"$ is a well-defined partial order with the following properties:
 \begin{enumerate}
 \item If a collection of distinct equivalence classes of 1-cells $\{[c_1],\dots,[c_k]\}$ has an upper bound, then it has a least upper bound and if $e_i$ is the unique edge in $[c_i]$ (that is, every cell in $[c_i]$ contains the edge $e_i$), then $e_{i_1}\cap e_{i_2}=\emptyset$ for $i_1\ne i_2.$
 \item For any $k$-cell $c$ of $ UD^n(\Gamma),$ there is a unique collection of equivalence classes of 1-cells $\{[c_1],[c_2],\dots,[c_k]\}$ having $[c]$ as its least upper bound.  
 \end{enumerate}
\end{lemma} 

Farley and Sabalka also introduced the idea of a ``cloud diagram" to represent an equivalence class.  These diagrams consist of a collection $E$ of edges and an indication of the number of vertices in each connected component of $\Gamma-E.$  The components of $\Gamma-E$ are called \emph{clouds.}  If $f(C)$ is the number of vertices in a cloud $C$, then 
\[
\sum_{\text{Clouds } C} f(C)=n-|E|.
\]
Figure \ref{fig:Clouds} gives three examples of cloud diagrams.  The cloud diagram on the left represents the class $[c],$ with $c$ as in Figure \ref{fig:CriticalCells}.

\Clouds

We will sometimes call the number of vertices in a cloud $C$ the \emph{value of $C$}.  Cloud diagrams also provide a convenient way to determine if $[d]\le[c].$  If this is the case, then the set of edges in the diagram for $[d]$ must be a subset of the set of edges in the diagram for $[c],$ which implies that each cloud in the diagram for $[c]$ must be contained in some cloud in the diagram for $[d].$ For each cloud $D$ in the diagram for $[d],$ the number of edges of $[c]$ which are contained in $D$ plus the sum of the values of the clouds of $[c]$ which are contained in $D$ must equal the value of $D.$  In particular, if the set of edges in $[c]$ equals the set of edges in $[d],$ then the diagrams for $[c]$ and $[d]$ have the same clouds, and $[c]$ and $[d]$ are comparable if and only if the values of each cloud are the same in both diagrams, in which case $[c]=[d].$ In Figure \ref{fig:Clouds}, the middle diagram is a cloud diagram for a class $[d]$ with $[d]\le[c]$ and the diagram on the right is a cloud diagram for a class $[d']$ which is comparable to neither $[c]$ nor $[d].$

Farley and Sabalka determine the structure of the cohomology ring $H^*(UD^n(\Gamma);\Z)$ by first constructing a space $\widehat{UD^n(\Gamma)}$ as follows.  For each equivalence class of $1$-cells $[c],$ let $S^1_{[c]}$ denote a circle with the usual cell structure consisting of a single open 1-cell $e^1_{[c]}$ and a single 0-cell.  Then, each open $k$-cell of the product $\prod_{[c]} S^1_{[c]},$ taken over all equivalence classes of 1-cells of $UD^n(\Gamma),$ is of the form $e^1_{[c_1]}\times\cdots\times e^1_{[c_k]},$ where we refrain from writing factors corresponding to 0-cells, and such a cell corresponds to a collection of equivalence classes of $1$-cells $\{[c_1],\dots,[c_k]\}.$  The space $\widehat{UD^n(\Gamma)}$ is obtained from $\prod_{[c]} S^1_{[c]}$ by removing open $k$-cells of the form $e^1_{[c_1]}\times\cdots\times e^1_{[c_k]}$ such that the corresponding collection $\{[c_1],\dots,[c_k]\}$ does not have an upper bound.  Then, each $k$-cell $\sigma$ in $\widehat{UD^n(\Gamma)}$ corresponds to a collection $\{[c_1],\dots,[c_k]\}$ which has an upper bound, and therefore a least upper bound $[c],$ and the cell $\sigma$ can be labeled by $[c].$  For each distinct equivalence class $[c]$ of cells in $UD^n(\Gamma),$ there is exactly one cell labeled $[c]$ in $\widehat{UD^n(\Gamma)}.$

Now, for a $k$-cell labeled $[c]$ in $\widehat{UD^n(\Gamma)},$ let $\hat\phi_{[c]}$ denote the $k$-cocycle in $C^*(\widehat{UD^n(\Gamma)};\Z)$ defined by 
\[
\hat\phi_{[c]}([c'])
=\begin{cases}
1,&\text{if }[c']=[c]\\
0,&\text{if }[c']\ne[c].
\end{cases}
\]
We have the following description of $H^*(\widehat{UD^n(\Gamma)};\Z):$
\begin{thm}{\rm\cite{FarleySabalkaCohomology},\cite{FarleyCohomology}}\quad \label{thm:UDnHatCohomology}
Let $\{[c_1],\dots,[c_M]\}$ be the collection of all equivalence classes of 1-cells in $UD^n(\Gamma).$  The cohomology ring $H^*(\widehat{UD^n(\Gamma)};\Z)$ is isomorphic to the quotient ring
\[
\Lambda[[c_1],\dots,[c_M]]/I,
\]
where $\Lambda[[c_1],\dots,[c_M]]$ is the integral exterior ring generated by the collection of all equivalence classes of 1-cells, and $I$ is the ideal generated by products $[c_{i_1}]\cdot[c_{i_2}]\cdots[c_{i_k}]$ such that the collection $\{[c_{i_1}],[c_{i_2}],\dots,[c_{i_k}]\}$ does not have an upper bound.

The isomorphism $H^*(\widehat{UD^n(\Gamma)})\to \Lambda[[c_1],\dots,[c_M]]/I$ sends $\hat{\phi}_{[c]}$ to $[c_1]\cdot[c_2]\cdots[c_k],$ where $\{[c_1],[c_2],\dots,[c_k]\}$ is the unique collection of equivalence classes of 1-cells which has $[c]$ as its least upper bound, arranged so that $\iota(e_i)<\iota(e_{i+1})$ for each $i$, where $e_i$ the unique edge in $[c_i].$
\end{thm}
The isomorphism in Theorem \ref{thm:UDnHatCohomology} depends on a choice of orientations of the cells and an ordering of the factors in $\widehat{UD^n(\Gamma)}.$  The details are given in \cite{FarleySabalkaCohomology} and \cite{FarleyCohomology}, but will be omitted here.

Similarly, for each equivalence class $[c]$ of $k$-cells of $UD^n(\Gamma),$ define a cellular cocycle $\phi_{[c]}\in C^*(UD^n(\Gamma);\Z)$ by
\[
\phi_{[c]}(c')=\begin{cases}1,&\text{if }c'\sim c\\
0,&\text{otherwise} .
\end{cases}
\]
These cocycles will be called standard cocycles, and if there is a (unique) critical cell in $[c],$ then $\phi_{[c]}$ is called a critical cocycle.  Since standard cocycles are determined by equivalence classes, cloud diagrams can also be used to describe standard cocycles.   

\begin{thm}{\rm\cite{FarleySabalkaCohomology},\cite{FarleyCohomology}}\quad \label{thm:cohomology} 
\begin{enumerate}
\item There is a well-defined map $q\co UD^n(\Gamma)\to \widehat{UD^n(\Gamma)},$ and the induced homomorphism $q^*\co C^*(\widehat{UD^n(\Gamma)})\to C^*(UD^n(\Gamma))$ sends the cocycle $\hat\phi_{[c]}$ to the standard cocycle $\phi_{[c]}.$ 
\item  The collection of critical cocycles represents a basis for $H^*(UD^n(\Gamma);\Z).$
\item  For any cell $c,$ we have $\phi_{[c]}^2=0.$
\item  If $c$ is a $k$-cell, and $\{[c_1],\dots,[c_k]\}$ is the unique collection of equivalence classes of 1-cells with $[c]$ as its least upper bound, arranged so that $\iota(e_i)<\iota(e_{i+1})$ for each $i,$ where $e_i$ is the unique edge in $[c_i],$ then
\[
\phi_{[c_1]}\cdots\phi_{[c_k]}=\phi_{[c]}.
\]
\item If $\{[c_1],\dots,[c_j]\}$ is any collection of equivalence classes of cells with no upper bound, then
\[
\phi_{[c_1]}\cdots\phi_{[c_j]}=0.
\]
\end{enumerate}
\end{thm}
Here, for any equivalence class $[c],$ we use $\phi_{[c]}$ to denote both the standard cocycle $\phi_{[c]}$ and the cohomology class it represents.  It follows from the universal coefficient theorem that the same statements hold true for rational cohomology, where we identify $\phi_{[c]}\in H^*(UD^n(\Gamma);\Z)$ with the corresponding class in $H^*(UD^n(\Gamma);\Q).$  Again, Theorem \ref{thm:cohomology} depends on a choice of orientations of cells, but we will omit these details.

If $\{[c_1],\dots,[c_k]\}$ is a collection of equivalence classes of $1$-cells which has a least upper bound $[\omega],$ then, $\phi_{[c_1]}\cdots\phi_{[c_k]}=\phi_{[\omega]},$ and if $[\omega]$ contains a critical cell, then $\phi_{[\omega]}$ is a critical cocycle, and therefore represents a basis element in $H^*(UD^n(\Gamma)).$  If $[\omega]$ does not contain a critical cell, then $\phi_{[\omega]}$ is cohomologous to a linear combination of critical cocycles.  In \cite{FarleyCohomology}, Farley gives a procedure to rewrite the cohomology class of $\phi_{[\omega]}$ in terms of critical cocycles, which we recall here.

If $[c]$ is an equivalence class of $k$-cells and $e$ is an edge of $[c]$ with $\iota(e)=v$ and $\deg(\tau(e))\le2,$ and $\mathcal{C}$ is the cloud diagram for the standard $k$-cocycle $\phi_{[c]},$ then a $(k-1)$-dimensional cochain $\mathcal{R}_{\mathcal{C},v}$ is defined as follows.  The support of $\mathcal{R}_{\mathcal{C},v}$ consists of $(k-1)$-cells $c_0$ such that
\begin{enumerate}[(i)]
\item $E(c_0)=E(c)-\{e\}$ (where for any cell $\sigma$ of $UD^n(\Gamma),\ E(\sigma)$ denotes the set of edges in $\sigma),$
\item if $C$ is any component of $\Gamma-E(c),$ other than the component $C_\iota$ which falls in the 0-direction from $v$ or the component $C_\tau$ which is adjacent to $\tau(e),$ then the number of vertices of $c_0$ in $C$ equals the number of vertices of $c$ in $C,$
\item the number of vertices of $c_0$ in $C_\iota\cup\{\iota(e)\}$ equals the number of vertices of $c$ in $C_\iota,$ and 
\item the number of vertices of $c_0$ in $C_\tau\cup\{\tau(e)\}$ is one more than the number of vertices of $c$ in $C_\tau.$  
\end{enumerate}
For each cell $c_0$ that satisfies these conditions, put $\mathcal{R}_{\mathcal{C},v}(c_0)=1.$  Let $E(\mathcal{R}_{\mathcal{C},v})$ denote the set of edges in any cell in the support of $\mathcal{R}_{\mathcal{C},v}.$  The cochain $\mathcal{R}_{\mathcal{C},v}$ can be described with a cloud diagram, where the union of the clouds around $v$ forms a connected component of $\Gamma-E(\mathcal{R}_{\mathcal{C},v}).$  For example, the left side of Figure \ref{fig:RCv} gives a cloud diagram $\mathcal{C}$ for a standard cocycle $\phi_{[c]}$ (which is not a critical cocycle), and the right side gives the cloud diagram for the cochain $\mathcal{R}_{\mathcal{C},v},$ where $v=\iota(e).$  Here, we have emphasized the clouds by indicating them with dotted lines.  For each cloud $C$ in the diagram for $\mathcal{R}_{\mathcal{C},v},$ let $f_{\mathcal{C},v}(C)$ denote the number of vertices in $C.$
\RCv

Farley shows that up to sign, the coboundary $\delta(\mathcal{R}_{\mathcal{C},v})$ is given by the following, where $d(v)$ denotes the degree of $v$:
\[
\delta(\mathcal{R}_{\mathcal{C},v})=\sum_{i=1}^{d(v)-1}\Theta_{\mathcal{C},v,\tau,i}-\sum_{i=1}^{d(v)-1}\Theta_{\mathcal{C},v,\iota,i}
\]
where $\Theta_{\mathcal{C},v,\tau,i}$ is the standard $k$-cocycle $\phi_{[c']}$ such that $E(c')=E(\mathcal{R}_{\mathcal{C},v})\cup\{e'_i\},$ where $e'_i$ is the edge in direction $i$ from $v,$ and if $C_i$ is the cloud in direction $i$ from $v$ in the cloud diagram for $\phi_{[c']},$  then the number of vertices in any cloud $C'$ in the cloud diagram for $\phi_{[c']}$ is given by 
\[
f_{[c']}(C')=
\begin{cases}
f_{\mathcal{C},v}(C'),&\text{if }C'\ne C_i\\
f_{\mathcal{C},v}(C')-1,&\text{if }C'=C_i.
\end{cases}
\]
Note there is a slight abuse of notation, in the sense that the clouds in the diagram for $\mathcal{R}_{\mathcal{C},v}$ are slightly different than those in the diagram for $\phi_{[c']}.$  For example, the cloud in direction 0 from $v$ in the diagram for $\mathcal{R}_{\mathcal{C},v}$ includes $\iota(e)$, whereas the cloud in direction 0 from $v$ in the diagram for $\phi_{[c']}$ does not include this vertex.  This should cause no confusion.  It is possible that $f_{[c']}(C')$ is negative; if this is the case, then $\phi_{[c']}$ is defined to be zero.  For example, Figure \ref{fig:ThetaTau} gives the sum $\sum_{i=1}^{d(v)-1}\Theta_{\mathcal{C},v,\tau,i}$ with $\mathcal{C}$ as in Figure \ref{fig:RCv}; here, $\Theta_{\mathcal{C},v,\tau,1}$ is zero.

\ThetaTau

Similarly, $\Theta_{\mathcal{C},v,\iota,i}$ is the standard cocycle $\phi_{[c'']},$ where $E(c'')=E(\mathcal{R}_{\mathcal{C},v})\cup\{e'_i\},$ and
\[
f_{[c'']}(C')=
\begin{cases}
f_{\mathcal{C},v}(C'),&\text{if }C'\ne C_0\\
f_{\mathcal{C},v}(C')-1,&\text{if }C'=C_0.
\end{cases}
\]
Here $C_0$ is the cloud in direction 0 from $v$ in the diagram for $\phi_{[c'']}.$  Figure \ref{fig:ThetaIota} gives the sum $\sum_{i=1}^{d(v)-1}\Theta_{\mathcal{C},v,\iota,i}$ with $\mathcal{C}$ as in Figure \ref{fig:RCv}.

\ThetaIota

Denote by $\widehat{\delta\mathcal{R}_{\mathcal{C},v}}$ the unique cochain in $C^*(\widehat{UD^n\Gamma})$ which maps to $\delta\mathcal{R}_{\mathcal{C},v}$.  If $[\omega]$ is an equivalence class of $k$-cells which does not contain a critical cell, then by the classification of critical cells, there must be some edge $e$ such that $e$ is order-respecting in every cell in $[\omega].$  In this case, call $e$ a \emph{bad edge} of $[\omega].$

\begin{thm}{\rm\cite{FarleyCohomology}\qua}\label{thm:Kernel}
Let $J\subset H^*(\widehat{UD^n(\Gamma)};\Z)$ denote the ideal generated by the classes $\widehat{\delta\mathcal{R}_{\mathcal{C},v}}$ where $\mathcal{C}$ is a cloud diagram containing at least one bad edge $e$ and $v=\iota(e).$  Then, we have
\[
H^*(UD^n\Gamma;\Z)\cong H^*(\widehat{UD^n\Gamma};\Z)/J
\]
\end{thm}

In the proof of Lemma \ref{lemma:products}, we will be interested in writing the cohomology class of a standard cocycle $\phi_{[\omega]}$ as a linear combination of basis elements (i.e. cohomology classes of critical cocycles), and comparing $\phi_{[\omega]}$ with these basis elements.  If $[\omega]$ contains a critical cell, then $\phi_{[\omega]}$ itself represents a basis element.  If $[\omega]$ does not contain a critical cell, the coboundaries $\delta\mathcal{R}_{\mathcal{C},v}$ give a way to rewrite the cohomology class of $\phi_{[\omega]}$ in terms of critical cocycles as follows.  Let $\mathcal{C}$ be the cloud diagram for $\phi_{[\omega]}.$  The class $[\omega]$ necessarily contains at least one bad edge $e$ which falls in direction $d_0$ from some vertex $v.$  We will again assume here that $\tau(e)$ has degree less than 3 (the case in which $\deg(\tau(e))\ge 3$ is addressed in \cite{FarleyCohomology}, but will not be needed here). Since $e$ is a bad edge, it must be the case that for all $0<i<d_0,$ we have $f_{[\omega]}(C_i)=f_{\mathcal{C},v}(C_i)=0,$ so the first non-zero term in $\sum \Theta_{\mathcal{C},v,\tau,i}$ is $\phi_{[\omega]},$ and any other non-zero term (if there are any) is of the form $\phi_{[\omega']},$ where $[\omega']$ contains an edge in direction $j>d_0$ from $v.$  Any non-zero term in $\sum\Theta_{\mathcal{C},v,\iota,i}$ (if there are any) is of the form $\phi_{[\omega''],}$ where $[\omega'']$ is an equivalence class of $k$-cells with the property that if $\{[\omega_1],\dots,[\omega_k]\}$ and $\{[\omega''_1],\dots,[\omega''_k]\}$ are the unique collections of equivalence classes of 1-cells which have $[\omega]$ and $[\omega'']$ as their respective least upper bounds, then there is some $i$ and $j$ such that $[\omega_i]$ contains the edge $e$ and $[\omega''_j]$ contains an edge with initial point at $v=\iota(e),$ but $f_{[\omega_j'']}(C_0)<f_{[\omega_i]}(C_0).$ Here, $C_0$ is the cloud in direction 0 from $v$ in the diagrams for $[\omega_i]$ and $[\omega''_j].$

Therefore, on the level of cohomology, we have $\phi_{[\omega]}=-A+B,$ where $A$ is a sum of standard cocycles $\phi_{[\omega']}$ such that $[\omega']$ has an edge $e'$ whose initial point is $v$, but $e'$ falls in a direction from $v$ greater than that in which $e$ falls, and $B$ is a sum of standard cocycles $\phi_{[\omega'']}$ as above.  It is possible that some of the terms in $A$ or $B$ are again standard cocycles corresponding to classes which do not contain critical cells (such as the three standard cocycles described in Figure \ref{fig:ThetaIota}), but if this is the case, we may rewrite each such cocycle using the procedure above.  Farley shows that this process eventually terminates and after repeatedly applying the procedure, we may write $\phi_{[\omega]}=\Sigma,$ where $\Sigma$ is a linear combination of critical cocycles.  Suppose $\phi_{[\widetilde{\omega}]}$ is a critical cocycle which appears in $\Sigma.$  Let $\{[\omega_1],\dots,[\omega_k]\}$ and $\{[\widetilde{\omega}_1],\dots,[\widetilde{\omega}_k]\}$ be the unique collections of equivalence classes of 1-cells which have $[\omega]$ and $[\widetilde{\omega}]$ as their respective least upper bounds.  We wish to compare the classes in each collection.  

First, note that the cloud diagram for any term appearing in $\delta(\mathcal{R}_{\mathcal{C},v})$ must contain an edge whose initial point is $\iota(e)=v$ and every edge $f\ne e$ in $[\omega]$ is also an edge in the cloud diagram for each term in $\delta(\mathcal{R}_{\mathcal{C},v}).$  Therefore, for any edge $e'$ in $[\omega],$ the cloud diagram of any term in $\Sigma$ must contain an edge who initial point is $\iota(e')$.  In other words, for each $i,$ there is some $s$ such that if $e'$ is the unique edge in $[\omega_i]$ and $e''$ is the unique edge in $[\widetilde{\omega}_s],$ then $\iota(e')=\iota(e'').$  Let $C_0$ be the cloud in direction 0 from $v$ in the cloud diagrams for $[\omega_i]$ and $[\widetilde{\omega}_s].$  At no point in the rewriting process do we add more vertices to the cloud $C_0.$  If $e'$ is a bad edge in $[\omega],$ then either (i) $f_{[\omega_i]}(C_0)=f_{[\widetilde{\omega}_s]}(C_0)$ and $\tau(e'')>\tau(e')$, or (ii) $f_{[\omega_i]}(C_0)>f_{[\widetilde{\omega}_s]}(C_0).$  If $e'$ is not a bad edge in $[\omega],$ then it may or may not become a bad edge at some stage of the rewriting process.  If $e'$ never becomes a bad edge, then we must have $e''=e'$ and $f_{[\omega_i]}(C_0)=f_{[\widetilde{\omega}_s]}(C_0).$  If $e'$ does become bad, then as above, either (i) $f_{[\omega_i]}(C_0)=f_{[\widetilde{\omega}_s]}(C_0)$ and $\tau(e'')>\tau(e'),$ or (ii) $f_{[\omega_i]}(C_0)>f_{[\widetilde{\omega}_s]}(C_0).$ 

Note that if $D_0$ is the cloud in the 0-direction from $v$ in the diagram for $[\omega]$ (so $D_0$ is contained in $C_0$), then it is possible that vertices are added to $D_0$ in the rewriting process if some edge with an initial point on the geodesic from $\iota(e')$ to $\ast$ becomes bad at some stage, but for each vertex added to $D_0,$ there must be some other cloud $D_1$ contained in $C_0$ which loses a vertex.  The observations in the preceding paragraph are similar to Farley's notion of the \emph{rank} of a cell $c$ defined in \cite{FarleyCohomology}.  

The coboundaries $\delta\mathcal{R}_{\mathcal{C},v}$ illustrate the complicated nature of the cohomology ring $H^*(UD^n\Gamma).$  The delicacy of the ring structure is studied further in \cite{SabalkaRigid}, for example, but the above is sufficient for what follows.

\section{Motion planning of configuration spaces of trees}\label{sec:treeMP}

Before proving Theorem \ref{thm:Main}, we first mention some of the tools for determining the topological complexity of any space.

\begin{thm}{\rm \cite{FarberTC}}\qua \label{thm:htpyinv}
$TC(X)$ is homotopy-invariant.  That is, if $X$ and $Y$ are homotopic, then $TC(X)=TC(Y).$
\end{thm}
Since we assume $\Gamma$ is sufficiently subdivided, the spaces $C^n(\Gamma)$ and $UC^n(\Gamma)$ are homotopic to $D^n(\Gamma)$ and $UD^n(\Gamma),$ respectively, by Theorem \ref{thm:defretract}, so Theorem \ref{thm:htpyinv} allows us to work with $D^n(\Gamma)$ and $UD^n(\Gamma)$ to determine $TC(C^n(\Gamma))$ and $TC(UC^n(\Gamma)).$  The next theorem gives an upper bound for $TC(X)$ based on the dimension of $X:$

\begin{thm}{\rm\cite{FarberTC}}\qua \label{thm:upperbound}
Let $X$ be any CW complex.  Then, we have the upper bound $TC(X)\le 2\cdot\mathrm{dim}(X)+1.$
\end{thm}
Finally, there is a cohomological lower bound for $TC(X).$  Before stating the theorem, we introduce some definitions.  Let $\mathbf{k}$ be a field, and consider the cup product
\[
\smile\co H^*(X;\mathbf{k})\otimes H^*(X;\mathbf{k})\to H^*(X;\mathbf{k}).
\]
Let $Z(X)\subset H^*(X;\mathbf{k})\otimes H^*(X;\mathbf{k})$ be the kernel of this homomorphism, called the ideal of \emph{zero-divisors} of $H^*(X;\mathbf{k}).$  The tensor product $H^*(X;\mathbf{k})\otimes H^*(X;\mathbf{k})$ has a multiplication given by $(\alpha\otimes\beta)(\alpha'\otimes\beta')=(-1)^{|\alpha'|\cdot|\beta|}\alpha\alpha'\otimes\beta\beta',$ where $|x|=j$ if $x\in H^j(X;\mathbf{k}).$  The \emph{zero-divisors-cup-length} of $H^*(X;\mathbf{k})$ is the largest $i$ such there are elements $a_1,\dots,a_i\in Z(X)$ with $a_1\cdots a_i\ne0.$ 

\begin{thm}{\rm\cite{FarberTC}}\qua \label{thm:lowerbound}
$TC(X)$ is greater than the zero-divisors-cup-length of $H^*(X;\mathbf{k}).$
\end{thm}

Now, we establish the upper bounds in Theorem \ref{thm:Main}:
\begin{lemma}\label{lemma:MotionPlanning}
Let $k=\min\left\{\left\lfloor\frac{n}{2}\right\rfloor,m(\Gamma)\right\}.$  Then, $TC(UD^n(\Gamma))\le 2k+1.$ 
\end{lemma}
\begin{proof}
This is immediate from the second statement of Theorem \ref{thm:kcells} and Theorems \ref{thm:htpyinv} and \ref{thm:upperbound}, but we will give a (fairly) explicit motion planning algorithm which realizes this upper bound.  This algorithm is similar to the one given by Farber in \cite{FarberCollisionFree}.  Let $a_1,\dots,a_n$ be the first $n$ vertices in a sufficiently subdivided tree $\Gamma$ (so $a_1=\ast$).  The fact that $\Gamma$ is sufficiently subdivided implies that only possibly $a_n$ is essential.  Consider a point $\mathbf{x}=\{x_1,\dots,x_n\}\in UD^n(\Gamma).$  If $x_i$ falls on a vertex $v$, let $f(x_i)$ be the number assigned to $v$ in the ordering above.  If $x_i$ falls on the interior of an edge $e,$ let $f(x_i)$ be the number assigned to $\iota(e).$   Since the order in which the $x_i$ appear in $\mathbf{x}$ is irrelevant, we may assume without loss of generality that $f(x_i)< f(x_{i+1})$ for all $i.$  Define a map $\sigma_{\mathbf{x}}\co[0,1]\to UD^n(\Gamma)$ as follows.  During the interval $\left[\frac{i-1}{n},\frac{i}{n}\right],\ \sigma_{\mathbf{x}}$ is the path which moves $x_i$ along the geodesic to $a_i$ at constant speed, and keeps all other $x_j$ fixed.  The choice of ordering of the vertices makes this a valid path in $UD^n(\Gamma)$ (i.e. there is at least a full open edge between any two components of $\sigma_{\mathbf{x}}(t)$ at any given time $t$).  Each $\sigma_{\mathbf{x}}$ is clearly continuous.  Define the section $s\co UD^n(\Gamma)\times UD^n(\Gamma)\to P(UD^n(\Gamma))$ by $s(\mathbf{x},\mathbf{y})=\sigma_{\mathbf{x}}\overline{\sigma_{\mathbf{y}}},$ the path $\sigma_{\mathbf{x}}$ followed by the reverse of $\sigma_{\mathbf{y}}.$  

This is not continuous on $UD^n(\Gamma)\times UD^n(\Gamma).$ If some $x_i$ (or $y_i$) falls on an endpoint $\tau(e)$ of some edge $e,$ a slight perturbation of $x_i$ (or $y_i$) may cause it to fall on the interior of $e,$ which will alter the numbering of the elements in $\mathbf{x}$  (or $\mathbf{y}$), which can lead to a very different path if $\iota(e)$ is essential.  So, we wish to examine the sets on which $s$ is continuous.  For a collection $E$ of edges, let $S_E$ be the set of points $\mathbf{x}=\{x_1,\dots,x_n\}\in UD^n(\Gamma)$ with the property that the interior of each edge $e\in E$ contains (exactly) one $x_i$ and no $x_j$ falls on the interior of any edge not in $E$ (so such an $x_j$ falls on a vertex which is not the endpoint of any $e\in E$).  The function $s$ is continuous on each $S_E\times S_{E'}.$  Now, let 
\[
S_i=\bigcup_{|E|=i}S_E.
\]
If $|E|=|E'|$ and $E\ne E',$ then a sequence of points in $S_{E}$ cannot converge to a point in $S_{E'},$ so that $\overline{S_E}\cap S_{E'}=\emptyset,$ and similarly $S_E\cap \overline{S_{E'}}=\emptyset.$  In other words, $S_i$ is a topologically disjoint union of the sets $S_E,$ and then for each fixed $i$ and $j,$ the set $S_i\times S_j$ is a topologically disjoint union of sets on which $s$ is continuous, so $s$ is continuous on $S_i\times S_j.$  Now, a sequence of points in $S_i$ may converge to a point in $S_{i'}$ for some $i'<i,$ but no sequence of points in $S_{i}$ can converge to a point in $S_{i'}$ if $i'>i,$ so the sets 
\[
U_l=\bigcup_{i+j=l}S_i\times S_j
\]
are again topologically disjoint unions of sets on which $s$ is continuous, so $s$ is continuous on $U_l.$  

The sets $U_0,\dots,U_{2k}$ cover $(UD^n(\Gamma))'_k\times(UD^n(\Gamma))'_k,$ since at most $k$ points can fall on the interior of an edge in either factor (see Theorem \ref{thm:kcells}).  They are not necessarily open, but each $U_l$ can be replaced with an open set $U'_l$ which allows each $x_i$ which falls on a vertex $v$ (and $x_i$ appears in a point $\mathbf{x}$ in the first component of $U_l$) to vary slightly away from $v$ (while keeping the point $\mathbf{x}$ in $UD^n(\Gamma)$), and defining $\sigma_{\mathbf{x},l}$ which is as above, except each of these $x_i$ is given the number for $v.$  This is well-defined, since each $v$ does not fall on any of the $l$ edges whose interiors are occupied by some $x_j$ in $\mathbf{x}$, so if $x_i$ falls on $v,$ a small perturbation of $x_i$ will not cause it to fall on the interior of any of those $l$ edges.  Similar modifications are made in the second component.  Define $s'_l\co U'_l\to P(UD^n(\Gamma))$ by  
\[
s'_l(\mathbf{x},\mathbf{y})=\sigma_{\mathbf{x},l}\overline{\sigma}_{\mathbf{y},l}
\]
This is continuous.  

If the map $H\co[0,1]\times UD^n(\Gamma)\to UD^n(\Gamma)$ is a deformation retraction from $UD^n(\Gamma)$ to $(UD^n(\Gamma))'_k$ with $H(0,-)$ equaling the identity map, then, $H_{\mathbf{x}}:=H(-,\mathbf{x})$ is a path from a point $\mathbf{x}\in UD^n(\Gamma)$ to some point in $(UD^n(\Gamma))'_k,$ which varies continuously with $\mathbf{x}.$  If $V_l=(H(1,-)\times H(1,-))^{-1}(U'_l),$ then $\{V_0,\dots,V_{2k}\}$ is an open cover of $UD^n(\Gamma)\times UD^n(\Gamma),$ and the section $s_l\co V_l\to P(UD^n(\Gamma))$ given by 
\[
s_l(\mathbf{x},\mathbf{y})=H_{\mathbf{x}}s'_l(\mathbf{x_1},\mathbf{y_1})\overline{H}_{\mathbf{y}},
\]
where $\mathbf{x_1}=H_\mathbf{x}(1)$ and $\mathbf{y_1}=H_\mathbf{y}(1),$ is continuous on each $V_l,$ for $l=0,\dots,2k,$ completing the proof.
\end{proof}

An additional step can be added to this algorithm to give the same upper bound for the ordered configuration spaces.  Again, this approach is essentially the one described in \cite{FarberCollisionFree}.
\begin{cor}\label{cor:OrderedUB}
Let $\Gamma$ be a tree with $m(\Gamma)\ge1,$ and let $k=\min\left\{\left\lfloor\frac{n}{2}\right\rfloor,m(\Gamma)\right\}.$  Then, $TC(D^n(\Gamma))\le 2k+1.$
\end{cor}
\begin{proof}
Let $a_1,\dots,a_n$ be the first $n$ vertices of $\Gamma,$ and let $\mathbf{a}=\{a_1,\dots,a_n\}\in UD^n(\Gamma).$  For $l=0,\dots,2k,$ let $V_l,\ \sigma_{\mathbf{x},l}$ and $H_{\mathbf{x}}$ be as in the proof of Lemma \ref{lemma:MotionPlanning}.  Let $p_1\co UD^n(\Gamma)\times UD^n(\Gamma)\to UD^n(\Gamma)$ be the projection of the first component, and let $V_l^1=p_1(V_l),$ so that if $\mathbf{x}\in V_l^1,$ then $H_{\mathbf{x}}\sigma_{\mathbf{x_1},l}$ (with $\mathbf{x_1}=H_\mathbf{x}(1)$) is a path from $\mathbf{x}$ to $\mathbf{a},$ and this path varies continuously as $\mathbf{x}$ varies in $V_l^1.$   If $\pi$ is the covering space projection $\pi\co D^n(\Gamma)\to UD^n(\Gamma),$  and $\widetilde{\mathbf{x}}\in \pi^{-1}(\mathbf{x})\subset D^n(\Gamma),$ with $\mathbf{x}\in V_l^1,$ let $\widetilde{\sigma}_{\widetilde{\mathbf{x}},l}$ be the unique lift of $H_\mathbf{x}\sigma_{\mathbf{x_1},l}$ which satisfies $\widetilde{\sigma}_{\widetilde{\mathbf{x}},l}(0)=\widetilde{\mathbf{x}}.$  This is a path from $\widetilde{\mathbf{x}}$ to some point $\widetilde{\mathbf{a}}\in\pi^{-1}(\mathbf{a})$ which varies continuously as $\widetilde{\mathbf{x}}$ varies in $\widetilde{V}_l^1=\pi^{-1}(V_l^1).$  Define $\widetilde{\sigma}_{\widetilde{\mathbf{y}},l}$ similarly if $\widetilde{\mathbf{y}}\in\pi^{-1}(\mathbf{y})$ for some $\mathbf{y}$ in the second component of $V_l.$

Now, given $\widetilde{\mathbf{a}},\widetilde{\mathbf{a}}'\in\pi^{-1}(\mathbf{a}),$ let $\rho_{\widetilde{\mathbf{a}},\widetilde{\mathbf{a}}'}$ be any path from $\widetilde{\mathbf{a}}$ to $\widetilde{\mathbf{a}}'$.   Such a path exists by Theorem \ref{thm:connectivity} since $\Gamma$ has at least one essential vertex.  The function 
\[
r\co \pi^{-1}(\mathbf{a})\times\pi^{-1}(\mathbf{a})\to P(D^n(\Gamma))
\]
 given by $r(\widetilde{\mathbf{a}},\widetilde{\mathbf{a}}')=\rho_{\widetilde{\mathbf{a}},\widetilde{\mathbf{a}}'}$ is continuous since the domain is a discrete space.  For each $l=0,\dots,2k,$ let $\widetilde{V}_l=(\pi\times\pi)^{-1}(V_l),$ and define the section $\widetilde{s}_l\co \widetilde{V}_l\to P(D^n(\Gamma))$ by 
\[
\widetilde{s}_l(\widetilde{\mathbf{x}},\widetilde{\mathbf{y}})=\widetilde{\sigma}_{\widetilde{\mathbf{x}},l}\,r(\widetilde{\sigma}_{\widetilde{\mathbf{x}},l}(1),\widetilde{\sigma}_{\widetilde{\mathbf{y}},l}(1))\,\overline{\widetilde{\sigma}_{\widetilde{\mathbf{y}},l}}.
\]
This is continuous on $\widetilde{V}_l,$ and $\{\widetilde{V}_0,\dots,\widetilde{V}_{2k}\}$ is an open cover of $D^n(\Gamma)\times D^n(\Gamma),$ completing the proof.
\end{proof}

In establishing the lower bound, the following will be useful:

\begin{lemma}\label{lemma:products}
Suppose $m(\Gamma)\ge1,\ UD^n(\Gamma)$ has the homotopy type of a $k$-dimensional CW complex, and $\Phi$ and $\Psi$ are critical $k$-cells of $UD^n(\Gamma).$  If $\{[c_1],\dots,[c_k]\}$ and $\{[d_1],\dots,[d_k]\}$ are the unique collections of equivalence classes of 1-cells having least upper bounds $[\Phi]$ and $[\Psi],$ respectively, and for all $i$ and $j$ we have $[c_i]\ne[d_j],$ then $TC(UD^n(\Gamma))$ and $TC(D^n(\Gamma))$ are greater than $2k.$
\end{lemma}

\begin{proof}
If necessary, rearrange the equivalence classes in $\{[c_1],\dots,[c_k]\}$ so that $\iota(e_i)<\iota(e_{i+1}),$ as above, and arrange the classes in $\{[d_1],\dots,[d_k]\}$ similarly.  Consider the zero divisors 
\[
\overline{\phi}_{[c_i]}=\phi_{[c_i]}\otimes1-1\otimes\phi_{[c_i]},\ \ \ \ \ \overline{\phi}_{[d_j]}=\phi_{[d_j]}\otimes1-1\otimes\phi_{[d_j]}
\]
in 
$Z(UD^n(\Gamma))\subset H^*(UD^n(\Gamma);\Q)\otimes H^*(UD^n(\Gamma);\Q)$ and their product
\begin{align}
&\biggl(\prod_{i=1}^k\overline{\phi}_{[c_i]}\biggr)\biggl(\prod_{j=1}^k\overline{\phi}_{[d_j]}\biggr)\label{ZeroDivisors}\\
&=\pm(\phi_{[c_1]}\cdots\phi_{[c_k]}\otimes\phi_{[d_1]}\cdots\phi_{[d_k]})\pm(\phi_{[d_1]}\cdots\phi_{[d_k]}\otimes\phi_{[c_1]}\cdots\phi_{[c_k]})\nonumber\\
&\hspace{10mm} +\text{\emph{other terms}}\nonumber\\
&=\pm\phi_{[\Phi]}\otimes\phi_{[\Psi]}\pm\phi_{[\Psi]}\otimes\phi_{[\Phi]}+\text{\emph{other terms.}}\label{ZeroDivisorsSum}
\end{align}

Since $UD^n(\Gamma)$ has the homotopy type of a $k$-dimensional complex, all products of more than $k$ 1-dimensional classes in $H^*(UD^n( \Gamma))$ are zero, so any non-zero term in \emph{other terms} must be of the form $\alpha\otimes\beta$ where $\alpha$ and $\beta$ are both degree-$k$ monomials in $\phi_{[c_1]},\dots,\phi_{[c_k]},\phi_{[d_1]},\dots,\phi_{[d_k]}.$  For a collection $I=\{[c_{i_1}],\dots,[c_{i_p}]\}$, we denote $\phi_{[c_{i_1}]}\cdots\phi_{[c_{i_p}]}$ by $\phi_I$ and denote the set $\{[c_i]:[c_i]\notin I\}$ by $\overline{I},$ and use analogous definitions for a collection $J=\{[d_{j_1}],\dots,[d_{j_r}]\}.$  Then, we can write $\alpha\otimes\beta=\phi_I\phi_J\otimes\phi_{\overline{I}}\phi_{\overline{J}},$ for some collections $I$ and $J$, with $I\ne\emptyset$ and $J\ne\emptyset$ (so that neither $I\cup J$ nor $\overline{I}\cup\overline{J}$ is either $\{[c_1],\dots,[c_k]\}$ or $\{[d_1],\dots,[d_k]\},$ by assumption), and $|I|+|J|=k.$

If either $I\cup J$ or $\overline{I}\cup\overline{J}$ does not have an upper bound, then $\alpha\otimes\beta=0.$  If both $I\cup J$ and $\overline{I}\cup\overline{J}$ have least upper bounds $[\omega]$ and $[\omega'],$ respectively, which both contain critical cells, then $\alpha\otimes\beta=\phi_{[\omega]}\otimes\phi_{[\omega']}$ is a basis element.  Since the collections of equivalence classes of 1-cells which have $[\Phi]$ and $[\Psi]$ as their least upper bounds are unique, and $I\cup J$ is neither $\{[c_1],\dots,[c_k]\}$ nor $\{[d_1],\dots,[d_k]\},$ we have $[\omega]\notin\{[\Phi],[\Psi]\}$ and similarly, $[\omega']\notin\{[\Phi],[\Psi]\},$ so $\alpha\otimes\beta\notin\{\phi_{[\Phi]}\otimes\phi_{[\Psi]},\phi_{[\Psi]}\otimes\phi_{[\Phi]}\}.$

If both $I\cup J$ and $\overline{I}\cup\overline{J}$ have upper least bounds $[\omega]$ and $[\omega'],$ but exactly one, say $[\omega]$, contains a critical cell, then $\alpha=\phi_{[\omega]},$ a basis element which is equal to neither $\phi_{[\Phi]}$ nor $\phi_{[\Psi]},$ as above.  We use the procedure following Theorem \ref{thm:Kernel} to rewrite the cocycle $\phi_{[\omega']}$ in terms of critical cocycles, arriving at $\alpha\otimes\beta=\phi_{[\omega]}\otimes \Sigma',$ where $\Sigma'$ is a linear combination of critical cocycles.  Note that since each edge $e$ in $[\omega']$ is an order-disrespecting edge in either $\Phi$ or $\Psi,$ the endpoint $\iota(e)$ must be essential, so this does not violate our assumption that $\tau(e)$ has degree less than 3 (we may need to further subdivide $\Gamma$ if $n=2$).  If $\Sigma'$ is non-zero, then $\alpha\otimes\beta$ can be written as a linear combination of basis elements of the form $\phi_{[\omega]}\otimes\phi_{[\widetilde{\omega}']},$ none of which are $\phi_{[\Phi]}\otimes\phi_{[\Psi]}$ or $\phi_{[\Psi]}\otimes\phi_{[\Phi]}.$  Similar statements hold if $[\omega']$ contains a critical cell.

Finally, suppose both $I\cup J$ and $\overline{I}\cup\overline{J}$ have upper least bounds $[\omega]$ and $[\omega'],$ but neither contains a critical cell.  We may write $\phi_{[\omega]}$ and $\phi_{[\omega']}$ as linear combinations of critical cocycles, as above:
\[
 \phi_{[\omega]}=\Sigma\hspace{1cm}
 \phi_{[\omega']}=\Sigma'.
\]
Since $[\omega]$ does not contain a critical cell, it must contain some bad edge $e.$  There is an equivalence class in $I\cup J$ which contains the edge $e,$ and this class is unique since $I\cup J$ has an upper bound (see Lemma \ref{lemma:partialorder}).  Suppose first that $[c_i]\in I$ is this class.   Consider a critical cocycle $\phi_{[\widetilde{\omega}]}$ appearing in $\Sigma,$ and let $\{[\widetilde{\omega}_1],\dots,[\widetilde{\omega}_k]\}$ be the unique collection of equivalence classes of 1-cells which has $[\widetilde{\omega}]$ as its least upper bound.  As in the discussion following Theorem \ref{thm:Kernel}, there must be some $s$ such that the unique edge $e'$ in $[\widetilde{\omega}_{s}]$ satisfies $\iota(e')=\iota(e),$ and either
\begin{equation}
f_{[c_i]}(C_0)=f_{[\widetilde{\omega}_s]}(C_0)\text{ and }\tau(e)<\tau(e'),\quad\text{or}\quad f_{[c_i]}(C_0)>f_{[\widetilde{\omega}_s]}(C_0).
\label{PhiNotInSigma}
\end{equation}
But, in either case, we have $[c_i]\ne[\widetilde{\omega}_s].$  Since $[c_i]$ is the only class in $\{[c_1],\dots,[c_k]\}$ which contains an edge whose initial point is $\iota(e),$ we have 
\[
\{[\widetilde{\omega}_1],\dots,[\widetilde{\omega}_k]\}\ne \{[c_1],\dots,[c_k]\},
\]
 so that, as above, $[\widetilde{\omega}]\ne[\Phi].$  Therefore the cocycle $\phi_{[\Phi]}$ does not appear in $\Sigma$ (by this, we mean that $\phi_{[\Phi]}$ does not appear with a non-zero coefficient in $\Sigma$).  
 
 Now, it is possible that $\phi_{[\widetilde{\omega}]}=\phi_{[\Psi]}.$  If this is the case, we have 
\[
\{[\widetilde{\omega}_1],\dots,[\widetilde{\omega}_k]\}=\{[d_1],\dots,[d_k]\},
\] 
so there must be some $j$ such that $[d_j]=[\widetilde{\omega}_s].$  Then, from (\ref{PhiNotInSigma}), we have either 
\begin{equation}
f_{[c_i]}(C_0)=f_{[d_j]}(C_0)\text{ and }\tau(e)<\tau(e'),\quad\text{or}\quad f_{[c_i]}(C_0)>f_{[d_j]}(C_0).
\label{PsiInSigma}
\end{equation}

We wish to show that in this case, $\phi_{[\Phi]}$ cannot appear in $\Sigma'.$  The edge $e$ is in the class $[c_i]$ and the edge $e'$ is in the class $[d_j]=[\widetilde{\omega}_s],$ and since $[c_i]\in I\cup J$ and $I\cup J$ contains an upper bound, it must be the case that $[d_j]\notin J.$  So, we have $[d_j]\in \overline{J},$ and therefore the edge $e'$ appears in $[\omega'].$  Let $\phi_{[\widetilde{\omega}']}$ be a critical cocycle appearing in $\Sigma',$ and let $\{[\widetilde{\omega}'_1],\dots,[\widetilde{\omega}'_k]\}$ be the unique collection of equivalence classes of 1-cells which has $[\widetilde{\omega}']$ as its least upper bound.  Because the edge $e'$ appears in $[\omega'],$ there must be some $t$ such that if $e''$ is the unique edge in $[\widetilde{\omega}'_t],$ then $\iota(e'')=\iota(e').$  Since the edge $e'$ may or may not be bad in $[\omega'],$ we have either 
\begin{equation}
f_{[d_j]}(C_0)=f_{[\widetilde{\omega}'_{t}]}(C_0)\text{ and }\tau(e')\le\tau(e''),\quad\text{or}\quad f_{[d_j]}(C_0)>f_{[\widetilde{\omega}'_{t}]}(C_0).
\label{PsiInSigma2}
\end{equation}
For the sake of contradiction, suppose that $\phi_{[\widetilde{\omega}']}=\phi_{[\Phi]}.$  Then, we have 
\[
\{[\widetilde{\omega}'_1],\dots,[\widetilde{\omega}'_k]\}=\{[c_1],\dots,[c_k]\},
\]
 so that $[\widetilde{\omega}'_t]=[c_{i_0}]$ for some $i_0.$ In particular, $e''$ is the unique edge in $[c_{i_0}],$ but since $\iota(e)=\iota(e')=\iota(e''),$ we must have $e''=e$ and $i_0=i.$ From (\ref{PsiInSigma2}), we have either
\[
f_{[d_j]}(C_0)=f_{[c_i]}(C_0)\text{ and }\tau(e')\le\tau(e),\quad\text{or}\quad f_{[d_j]}(C_0)>f_{[c_i]}(C_0),
\]
but, this contradicts (\ref{PsiInSigma}).  

So, we have shown that if $[c_i]$ is the class in $I\cup J$ which contains the bad edge $e,$ then it is impossible that $\phi_{[\Phi]}$ appears in $\Sigma,$ and if $\phi_{[\Psi]}$ appears in $\Sigma,$ it is impossible that $\phi_{[\Phi]}$ appears in $\Sigma'.$ By a symmetric argument, if $[d_{j'}]$ is the class in $I\cup J$ which contains $e,$ it follows that it is impossible that $\phi_{[\Psi]}$ appears in $\Sigma,$ and if $\phi_{[\Phi]}$ appears in $\Sigma,$ it is impossible that $\phi_{[\Psi]}$ appears in $\Sigma'.$  So, in either case, neither $\phi_{[\Phi]}\otimes\phi_{[\Psi]}$ nor $\phi_{[\Psi]}\otimes\phi_{[\Phi]}$ appears in the expansion of $\Sigma\otimes\Sigma'$ as a linear combination of simple tensors of critical cocycles.

Therefore (\ref{ZeroDivisorsSum}) may be written as a linear combination of basis elements of the tensor product $H^*(UD^n(\Gamma))~\otimes~H^*(UD^n(\Gamma)),$ with neither $\phi_{[\Phi]}\otimes\phi_{[\Psi]}$ nor $\phi_{[\Psi]}\otimes\phi_{[\Phi]}$ appearing in \emph{other terms}, and since $\Phi$ and $\Psi$ are distinct critical cells, 
\[
\pm\phi_{[\Phi]}\otimes\phi_{[\Psi]}\pm\phi_{[\Psi]}\otimes\phi_{[\Phi]}\ne0,
\]
so the entire sum is non-zero.   So, (\ref{ZeroDivisors}) is a nonzero product of $2k$ zero-divisors, and Theorem \ref{thm:lowerbound} establishes that $TC(UD^n(\Gamma))\ge2k+1,$ as desired.

The statement for $TC(D^n(\Gamma))$ follows from the fact that the map 
\[
\pi^*\co H^*(UD^n(\Gamma))\to H^*(D^n(\Gamma))
\]
induced by the covering space projection $\pi\co D^n(\Gamma)\to UD^n(\Gamma)$ is injective (see \cite[Proposition 3G.1]{Hatcher}, which states that any $N$-sheeted covering space projection given by a group action induces an injection in cohomology with coefficients in a field of characteristic 0).
\end{proof}

Note that the condition $[c_i]\ne[d_j]$ for all $i$ and $j$ is necessary for (\ref{ZeroDivisors}) to be non-zero.  If $[c_i]=[d_j]$ for some $i$ and $j,$ then $\phi_{[c_i]}=\phi_{[d_j]},$ so (\ref{ZeroDivisors}) contains the product 
\begin{align*}
\overline{\phi}_{[c_i]}\overline{\phi}_{[c_i]}&=(\phi_{[c_i]}\otimes1-1\otimes\phi_{[c_i]})(\phi_{[c_i]}\otimes1-1\otimes\phi_{[c_i]})\\
&=(\phi_{[c_i]}\otimes1)^2-(\phi_{[c_i]}\otimes1)(1\otimes\phi_{[c_i]})-(1\otimes\phi_{[c_i]})(\phi_{[c_i]}\otimes1)+(1\otimes\phi_{[c_i]})^2\\
&=\phi_{[c_i]}^2\otimes1-(-1)^0(\phi_{[c_i]}\otimes\phi_{[c_i]})-(-1)^1(\phi_{[c_i]}\otimes\phi_{[c_i]})+1\otimes\phi_{[c_i]}^2\\
&=0.
\end{align*}
Therefore, (\ref{ZeroDivisors}) is non-zero if and only if $[c_i]\ne[d_j]$ for any $i$ and $j.$  

We now prove Theorem \ref{thm:Main}.
\begin{proof} [Proof of Theorem \ref{thm:Main}]

Assume without loss of generality that $\Gamma$ is sufficiently subdivided for $n,$ so that the topological and discrete configuration spaces have the same homotopy type, and hence, the same topological complexity.  We will prove the results for the discrete configuration spaces.  The upper bounds, $TC(UD^n(\Gamma))\le2m+1,$ $TC(D^n(\Gamma))\le2m+1$ in statement \ref{nLarge} and $TC(UD^n(\Gamma))\le2q+1,$ $TC(D^n(\Gamma))\le2q+1,$ in statement \ref{nNotLarge} are given in Lemma \ref{lemma:MotionPlanning} and Corollary \ref{cor:OrderedUB}.  To establish the lower bounds, we will construct two $m$-cells for statement \ref{nLarge} and two $q$-cells for statement \ref{nNotLarge} that satisfy the conditions of Lemma \ref{lemma:products}.  

For statement \ref{nLarge}, first note that Theorem \ref{thm:kcells} implies $UD^n(\Gamma)$ is homotopic to an $m$-dimensional CW complex.  Let $v_1,\dots,v_m$ be (in order) the essential vertices of $\Gamma$ and let $A_1,\dots,A_k$ be a collection of oriented arcs which is allowable for the set of vertices of degree 3.  For each arc $A_l,$ let $a_l$ and $b_l$ be the initial and terminal points of $A_l,$ with respect to the orientation of $A_l.$  The endpoints need not be vertices of $\Gamma.$

We will construct an $m$-cell $\Phi$ as follows.  At each essential vertex $v_i,$ let $e_i$ be the edge in direction 2 from $v_i$ (so that $ \iota(e_i)=v_i),$ let $f_i$ be the edge in direction 1 from $v_i,$ and let $u_i=\tau(f_i).$  The labeling of the vertices forces $v_i=\iota(e_i)<u_i<\tau(e_i),$ so that $e_i$ is order-disrespecting in any cell containing $u_i.$  Furthermore, $u_i$ is blocked in any cell containing $e_i.$  Add each edge $e_i$ and each vertex $u_i$ to $\Phi.$  The edges $e_i$ determine a system of clouds of $\Gamma.$  

Now, if $n>2m$ (which must be the case, by assumption, if there are any vertices of degree 3), we must have $n\ge3,$ so since $\Gamma$ is sufficiently subdivided, if the endpoint $a_l$ falls on an edge $e_i,$ we may shrink or enlarge $A_l$ slightly so that the new initial point falls just beyond $\tau(e_i),$  without changing the fact that the collection $\{A_i\}_{i=1}^k$ is allowable for the collection of degree-3 vertices.  So, we may assume that each endpoint $a_l$ falls in one of the clouds of $\Gamma$ determined by the edges $e_i.$   Then, inductively, for $l=1,\dots,k,$ let $w_l$  be the minimal vertex in the cloud containing $a_l$ which we have not already included in $\Phi,$ and add this vertex to $\Phi.$ Then, $w_1,\dots,w_k$ are all blocked in $\Phi.$ Finally, let $w_{k+1},\dots,w_{n-2m}$ be the first $n-2m-k$ vertices of $\Gamma$ which are not already in $\Phi,$ and add them to $\Phi$ (so they are all blocked).  So, we have
\[
\Phi=\{e_{1},\dots,e_{m},u_1,\dots,u_m,w_{1},\dots,w_{n-2m}\},
\]
where the vertices $w_i$ do not appear if $n=2m.$  The cell $\Phi$ is critical.  Now, we will construct a critical $m$-cell $\Psi$ similarly.  If $v_i$ is a vertex of degree 3, let $e'_i=e_i$ and $u'_i=u_i.$  If $v_i$ is a vertex of degree greater than 3, let $e'_i$ be the edge in direction 3 from $v_i,$ let $f'_i$ be the edge in direction 2, and let $u'_i=\tau(f'_i),$ so $u'_i$ is blocked by $e'_i.$  As above, we have $v'_i=\iota(e'_i)<u'_i<\tau(e'_i),$ so $e'_i$ is order-disrespecting in any cell containing $u'_i.$  Add the edges $e'_i$ and vertices $u'_i$ to the cell $\Psi.$  Similar to the argument above, if $n>2m,$ we may assume each endpoint $b_l$ falls in one of the clouds of $\Gamma$ determined by the edges $e'_i.$  Inductively for $l=1,\dots,k,$ let $w'_l$  be the minimal vertex in the cloud containing $b_l$ which we have not already included in $\Psi,$ and add this vertex to $\Psi.$ Let $w'_{k+1},\dots,w'_{n-2m},$ be the first $n-2m-k$ vertices of $\Gamma$ which are not already in $\Psi,$ and add them in, so now we have
\[
\Psi=\{e'_{1},\dots,e'_{m},u'_1,\dots,u'_m,w'_{1},\dots,w'_{n-2m}\},
\]
where the vertices $w'_i$ do not appear if $n=2m.$  The cell $\Psi$ is critical.   Figure \ref{fig:PartOneA} gives an example for $m=11,\ k=3,$ and $n=27$. The figure on the left is a tree $\Gamma$ with 3 oriented arcs which are allowable for the collection of degree-3 vertices.  The orientations of the arcs are indicated with arrows.

\PartOneA

Now, let $\{[c_1],\dots,[c_m]\}$ be the unique collection of equivalence classes of 1-cells having $[\Phi]$ as its least upper bound.  The  equivalence class $[c_i]$ can be represented using a cloud diagram having a single edge (the edge $e_i$) in direction 2 from the essential vertex $v_i,$ whose degree is $t_i,$ and $t_i$ clouds.  Likewise, let $\{[d_1],\dots,[d_m]\}$ be the unique collection of equivalence classes of 1-cells having $[\Psi]$ as its least upper bound.  The equivalence class $[d_i]$ can be represented using a cloud diagram having a single edge (the edge $e'_i$) and $t_i$ clouds.  If $t_i\ge4,$ then $[c_i]\ne[d_i],$ since the edges $e_i$ and $e'_i$ differ.  

If $t_i=3,$ then the cloud diagram for $[c_i]$ will contain the single edge $e_i$ and three clouds $C_0,\ C_1,\ C_2,$ in the 0, 1, and 2 directions from $v_i.$   The cloud diagram for $[d_i]$ will have the same edge and clouds.  Suppose $f(C)$ is the value assigned to cloud $C$ in the diagram for $[c_i]$ and $g(C)$ is the value assigned to cloud $C$ in the diagram for $[d_i].$ For $\delta=0,1,2,$ let $m_\delta$ be the number of essential vertices in $C_\delta,$ and let $W_\delta=|\{p:w_p\in C_\delta\}|$ and $W'_\delta=|\{p:w'_p\in C_\delta\}|.$ Then $f(C_\delta)=2m_\delta+W_\delta+\lambda_\delta$ and $g(C_\delta)=2m_\delta+W'_\delta+\lambda_\delta,$ where $\lambda_0=\lambda_2=0$ and $\lambda_1=1.$  

Suppose the vertex $v_i$ falls on the interior of an arc $A_l$ (by assumption there is at least one such arc).  For each direction $\delta,$ if the arc $A_l$ is oriented towards $v_i$ in the direction $\delta,$ then $a_l$  must fall in $C_\delta,$ and $b_l$ must fall in a different cloud, so the vertex $w_l$ contributes 1 to $W_\delta,$ but $w'_l$ contributes 0 to $W'_\delta.$  On the other hand, if $A_l$ is oriented away from $v_i$ in the direction $\delta,$ then $b_l$  must fall in $C_\delta,$ and $a_l$ must fall in a different cloud so the vertex $w'_l$ contributes 1 to $W'_\delta,$ but $w_l$ contributes 0 to $W_\delta.$  If $v_i$ does not fall on $A_l,$ then both $w_l$ and $w'_l$ must be in the same cloud, so either $w_l$ and $w'_l$ contribute 1 to $W_\delta$ and $W'_\delta,$ respectively, or each contributes 0 to $W_\delta$ and $W'_\delta,$ respectively.  Finally, note that each $w_i$ and $w'_i$ for $i>k$ must fall in the same cloud as the basepoint $\ast$ in both $\Phi$ and $\Psi$ since $\Gamma$ is sufficiently subdivided, so in this case $w_i$ and $w'_i$ contribute 1 to $W_0$ and $W'_0,$ respectively.

In other words, the difference $W_\delta-W'_\delta$ is equal to the number of arcs oriented towards $v_i$ in the direction $\delta$ minus the number of arcs oriented away from $v_i$ in the direction $\delta.$  This is the number $\eta_\delta(v_i)$ in Definition \ref{defn:AllowableOrientations}, so by assumption, there must be at least one direction $\delta$ such that $W_\delta\ne W'_\delta.$  Furthermore, by the remarks following  the same definition, there must be directions $\delta$ and $\delta'$ such that $W'_\delta-W_\delta>0$ and $W'_{\delta'}-W_{\delta'}<0,$ and therefore, $g(C_\delta)>f(C_\delta)$ and $g(C_{\delta'})<f(C_{\delta'}).$  In this case, call $C_\delta$ and $C_{\delta'}$ positive and negative clouds, respectively, to reflect that $[d_i]$ has more vertices than $[c_i]$ in direction $\delta$ (from $v_i$) and less vertices in direction $\delta'.$  It is possible that either  all 3 clouds are categorized as either positive or negative or that one cloud remains uncategorized.  But, in either case, we have $[c_i]\ne[d_i].$  As an example, the Figure \ref{fig:PartOneB} shows the classes $[c_4]$ and $[d_4]$ with $\Gamma$ as in Figure \ref{fig:PartOneA}.

\PartOneB

Furthermore, it follows from this description that $[c_i]\ne[d_j]$ for any $i$ and $j.$  Indeed, if $[c_i]=[d_j],$ then both classes must contain a common edge. But, this can only happen if $i=j$ and $t_i=3,$ but we just saw $[c_i]\ne[d_i]$ in this case.   So, the cells $\Phi$ and $\Psi$ satisfy the hypotheses of Lemma \ref{lemma:products}, proving statement \ref{nLarge}. 

The construction for statement \ref{nNotLarge} is similar. Here, by Theorem \ref{thm:kcells}, $UD^n(\Gamma)$ is homotopic to a complex of dimension $q.$  Let $v_1,\dots,v_r$ be the vertices of degree greater than 3, and let $\tilde{v}_1,\dots,\tilde{v}_s$ be the vertices of degree 3.  We consider the 2 cases in the statement:

{\bf Case \ref{nSmall}, $s\ge2(q-r)$:} Let $R=\mathrm{min}\{r,q\},$ and for $i=1,\dots,R,$ let $e_i,\ u_i,\ e'_i,$ and $u'_i$ be as in the proof of statement \ref{nLarge} (so that $e_i,e'_i$ are order-disrespecting edges in directions 2 and 3 from $v_i$, and $u_i,u'_i$ are blocked vertices in directions 1 and 2).  If $r<q,$ for $i=1,\dots,2(q-r),$ let $\tilde{e}_i$ be the edge in direction 2 from $\tilde{v}_i,$ and let $\tilde{u}_i$ be the vertex on the edge in direction 1 from $\tilde{v}_i$ which forces $\tilde{e}_i$ to be order-disrespecting. If $\epsilon$=0, let
\begin{align*}
\Phi&=\{e_{1},\dots,e_{R},u_1,\dots,u_R,\tilde{e}_1,\dots,\tilde{e}_{q-r},\tilde{u}_1,\dots\tilde{u}_{q-r}\},\\
\Psi&=\{e'_{1},\dots,e'_{R},u'_1,\dots,u'_R,\tilde{e}_{q-r+1},\dots,\tilde{e}_{2(q-r)},\tilde{u}_{q-r+1},\dots,\tilde{u}_{2(q-r)}\},
\end{align*}
where the edges $\tilde{e}_i$ and vertices $\tilde{u}_i$ do not appear if $r\ge q.$  If $\epsilon=1,$ add the vertex $\ast$ to each cell. Figure \ref{fig:PartTwoA} gives an example for $\Gamma$ as in Figure \ref{fig:PartOneA} with $q=5,\ \epsilon=0.$  

\PartTwoA

For the same reasons as in the first part of the proof, the cells $\Phi$ and $\Psi$ are critical, and if $\{[c_1],\dots,[c_{q}]\}$ and $\{[d_1],\dots,[d_q]\}$ are the collections of equivalence classes of 1-cells having $[\Phi]$ and $[\Psi]$ as their least upper bounds, here it is clear that $[c_i]\ne[d_j]$ for any $i$ and $j$ since no edge in $\Phi$ is in $\Psi.$ 

{\bf Case \ref{nMedium}, $s<2(q-r)$:}  First, consider the $\epsilon=0$ case.  Let $\{A_i\}_{i=1}^k$ and $\mathcal{V}$ be as in the statement.  For $l=1,\dots,k,$ let $V_l$ and $V_l'$ be the initial and terminal vertices of the arc $A_l$ (again with respect to the orientation of $A_l$). Let $E_l$ be the edge in direction 2 from $V_l,$ and let $U_l$ be the blocked vertex in direction 1 from $V_l$ which forces $E_l$ to be order-disrespecting. Define $E'_l$ and $U'_l$ similarly (so that $E'_l$ is in direction 2 from $V'_l$ and $U'_l$ is in direction 1).  Note we must have $2k\le m-r'$ so $s+r=m\ge 2k+r',$ and also $r\ge r',$ so by assumption, we have $2q> s+2r=s+r+r\ge2k+r'+r\ge2k+2r',$ and therefore $q-r'-k>0.$

Let $\widehat{V}_1,\dots,\widehat{V}_{r'}$ be the vertices of degree greater than 3 which aren't endpoints of any arc $A_l.$  Define $\widehat{E}_i$ and $\widehat{U}_i$ analogously to the definitions of $E_l$ and $U_l.$  Let $\widehat{E}'_i$ be the edge in direction 3 from $\widehat{V}_i,$ and let $\widehat{U}'_i$ be the vertex in direction 2 which forces $\widehat{E}'_i$ to be order-disrespecting.  Let $\widetilde{V}_1,\dots,\widetilde{V}_{q-r'-k}$ be the first $q-r'-k$ vertices in $\mathcal{V}.$  Define $\widetilde{E}_i$ and $\widetilde{U}_i$ analogously to $E_i$ and $U_i.$  Let
\begin{align*}
\Phi&=\{E_1,\dots,E_k,U_1,\dots,U_k,\widehat{E}_1,\dots,\widehat{E}_{r'},\widehat{U}_1,\dots,\widehat{U}_{r'},\\
&\hspace{4cm}\widetilde{E}_1,\dots,\widetilde{E}_{q-r'-k},\widetilde{U}_1,\dots,\widetilde{U}_{q-r'-k}\},\\
\Psi&=\{E'_1,\dots,E'_k,U'_1,\dots,U'_k,\widehat{E}'_1,\dots,\widehat{E}'_{r'},\widehat{U}'_1,\dots,\widehat{U}'_{r'},\\
&\hspace{4cm}\widetilde{E}_1,\dots,\widetilde{E}_{q-r'-k},\widetilde{U}_1,\dots,\widetilde{U}_{q-r'-k}\}.
\end{align*}

Figure \ref{fig:PartTwoBi} gives an example for $q=8,\epsilon=0,\ r'=2,$ and $k=3.$  The set $\mathcal{V}$ consists of the vertices $\widetilde{V}_1,\widetilde{V}_2,$ and $\widetilde{V}_3,$ so that $|\mathcal{V}|=3=q-r'-k.$ 

\PartTwoBi

The cells $\Phi$ and $\Psi$ are critical.  As above, if $\{[c_1],\dots,[c_{q}]\}$ and $\{[d_1],\dots,[d_q]\}$ are the collections of equivalence classes of 1-cells having $[\Phi]$ and $[\Psi]$ as their least upper bounds, and if $[c_i]=[d_j]$ for some $i$ and $j,$ then $[c_i]$ and $[d_j]$ must have a common edge $e$ which can only happen if $e=\widetilde{E}_t$ for some $t,$ so in particular, $\widetilde{V}_t$ is of degree 3.  For such a $t$ and each $\delta=0,1,2,$ let $C_\delta$ be the cloud in direction $\delta$ from $\widetilde{V}_t,$ and let 
\begin{align*}
M_\delta&=|\{p:V_p\in C_\delta\}|, &M'_\delta&=|\{p:V'_p\in C_\delta\}|,\\
\widehat{M}_\delta&=|\{p:\widehat{V}_p\in C_\delta\}|, &\widetilde{M}_\delta&=|\{p:p\ne t,\widetilde{V}_p\in C_\delta\}|.
\end{align*}
Then, if $f(C_\delta)$ and $g(C_\delta)$ are the values of $C_\delta$ in the diagrams for $[c_i]$ and $[d_j],$ and $\lambda_\delta$ is as above, we have 
\begin{align*}
f(C_\delta)&=2(M_\delta+\widehat{M}_\delta+\widetilde{M}_\delta)+\lambda_\delta,\\
g(C_\delta)&=2(M'_\delta+\widehat{M}_\delta+\widetilde{M}_\delta)+\lambda_\delta,
\end{align*}
so that $f(C_\delta)-g(C_\delta)=2(M_\delta-M'_\delta).$  Arguments similar given to those in the first part of the proof shows that $M_\delta-M'_\delta=\eta_\delta(\widetilde{V}_t),$ so again we see that at least two of the clouds around $\widetilde{V}_t$ can be categorized as either positive or negative, and each of these clouds has a different value in the diagram for $[c_i]$ than it does in the diagram for $[d_j],$ so $[c_i]\ne[d_j].$  

For $\epsilon=1,$ construct $\Phi$ and $\Psi$ as above with the following modifications.  First, the vertices $\widetilde{V}_1,\dots,\widetilde{V}_{q-r'-k}$ are the first vertices in $\mathcal{W}'\cup\mathcal{W}$ (where $k=0$ and $r'=r$ if $s'\ge q-r$). Next, since $n=2q+1\ge3,$ we can again assume the initial and terminal endpoints of $A_0$ fall in a cloud in the collection of clouds determined by the edges in $\Phi$ and $\Psi,$ respectively.  Then, let $x$ be a vertex in the same cloud as the initial endpoint of $A_0$ (in the system of clouds determined by the edges of $\Phi$) such that $x$ is blocked in
\begin{align*}
\Phi&=\{E_1,\dots,E_k,U_1,\dots,U_k,\widehat{E}_1,\dots,\widehat{E}_{r'},\widehat{U}_1,\dots,\widehat{U}_{r'},\\
&\hspace{4cm}\widetilde{E}_1,\dots,\widetilde{E}_{q-r'-k},\widetilde{U}_1,\dots,\widetilde{U}_{q-r'-k},x\},
\end{align*}
and let $x'$ be a vertex in the same cloud as the terminal endpoint of $A_0$ (in the system of clouds determined by the edges of $\Psi$) such that $x'$ is blocked in
\begin{align*}
\Psi&=\{E'_1,\dots,E'_k,U'_1,\dots,U'_k,\widehat{E}'_1,\dots,\widehat{E}'_{r'},\widehat{U}'_1,\dots,\widehat{U}'_{r'},\\
&\hspace{4cm}\widetilde{E}_1,\dots,\widetilde{E}_{q-r'-k},\widetilde{U}_1,\dots,\widetilde{U}_{q-r'-k},x'\}.
\end{align*}
Here, we of course assume that $x$ is distinct from the other vertices in $\Phi$ and $x'$ is distinct from the other vertices in $\Psi.$  If $s'\ge q-r,$ the edges $E_i$ and $E'_i$ and the vertices $U_i$ and $U'_i$ do not appear in $\Phi$ and $\Psi.$  

For example, for $n=17,$ we can let the arcs $A_1,\ A_2,$ and $A_3$ be as in Figure \ref{fig:PartTwoBi}, so $\mathcal{W}=\{\widetilde{V}_1,\widetilde{V}_2,\widetilde{V}_3\},$ and trivially let the arc $A_0$ be the unique edge which has $\ast$ as an endpoint, so that $\mathcal{W}'=\emptyset.$  For a less trivial example, if $q=9,\ \epsilon=1,$ we can let $A_0$ be the arc $A_3$ in Figure \ref{fig:PartTwoBi} slightly enlarged so that its interior includes its two original endpoints labeled $V_3$ and $V_3'$ in Figure \ref{fig:PartTwoBi}, so that here now we have $k=2,\ r'=2.$  See Figure \ref{fig:PartTwoBii}.  The set $\mathcal{W}'$ consists of the vertices labeled $\widetilde{V}_1,\dots,\widetilde{V}_5,$ and $\mathcal{W}=\emptyset$ so that $|\mathcal{W}|=0=9-2-2-5=q-r'-k-s'.$  Alternatively, we can let $\mathcal{W}'$ consist of the vertices labeled $\widetilde{V}_3,\widetilde{V}_5$ and let $\mathcal{W}$ consist of the vertices labeled $\widetilde{V}_1,\widetilde{V}_2,\widetilde{V}_4$ in Figure \ref{fig:PartTwoBii}, so that $s'=2$ and $|\mathcal{W}|=3=9-2-2-2=q-r'-k-s'.$ 

\PartTwoBii

Again, the cells $\Phi$ and $\Psi$ are critical.  If $\{[c_1],\dots,[c_{q}]\}$ and $\{[d_1],\dots,[d_q]\}$ are the collections of equivalence classes of 1-cells having $[\Phi]$ and $[\Psi]$ as their least upper bounds, and $[c_i]=[d_j]$ for some $i$ and $j,$ then again it must be the case that $[c_i]$ and $[d_j]$ have a common edge $\widetilde{E}_t,$ so $\widetilde{V}_t$ is of degree 3.  Now, with the notation from the $\epsilon=0$ case, we have
\begin{align*}
f(C_\delta)&=2(M_\delta+\widehat{M}_\delta+\widetilde{M}_\delta)+\lambda_\delta+\mu_\delta,\\
g(C_\delta)&=2(M'_\delta+\widehat{M}_\delta+\widetilde{M}_\delta)+\lambda_\delta+\mu'_\delta,
\end{align*}
where $\mu_\delta=1$ if $x$ falls in $C_\delta,$ and $\mu_\delta=0$ if $x$ does not fall in $C_\delta,$ and similarly, $\mu'_\delta=1$ if $x'$ falls in $C_\delta,$ and $\mu'_\delta=0$ if $x'$ does not fall in $C_\delta.$  Since the endpoints of $A_0$ fall in clouds determined by the edges in $\Phi$ and $\Psi,$ the vertex $\widetilde{V}_t$ cannot be an endpoint, so it either does not fall on $A_0$ or it falls on the interior of $A_0.$  

If $\widetilde{V}_t$ does not fall on $A_0,$ so that we have $\widetilde{V}_t\in\mathcal{W},$ then $x$ and $x'$ are in the same cloud in the diagrams for $[c_i]$ and $[d_j],$ so that $\mu_\delta=\mu'_\delta$ for each $\delta,$ and therefore $f(C_\delta)-g(C_\delta)=2(M_\delta-M'_\delta)=2\eta_\delta(\widetilde{V}_t),$ and since the collection $\{A_i\}_{i=1}^k$ is allowable for $\mathcal{W},$ we can categorize at least one cloud as positive and one as negative as above, so that $[c_i]\ne[d_j].$  

If $\widetilde{V}_t$ falls on the interior of $A_0,$ then it is possible that $\widetilde{V}_t\in\mathcal{W}$ or $\widetilde{V}_t\in\mathcal{W}',$ but in either case, $x$ and $x'$ must fall in different clouds $C_{\delta_1}$ and $C_{\delta_2}.$ Then, we have 
\begin{align*}
f(C_{\delta_1})&=2(M_{\delta_1}+\widehat{M}_{\delta_1}+\widetilde{M}_{\delta_1})+\lambda_{\delta_1}+1\\
g(C_{\delta_1})&=2(M'_{\delta_1}+\widehat{M}_{\delta_1}+\widetilde{M}_{\delta_1})+\lambda_{\delta_1}+0,
\end{align*}
so now $f(C_{\delta_1})-g(C_{\delta_1})=2(M_{\delta_1}-M'_{\delta_1})+1,$ which is odd and therefore non-zero.  Likewise, $f(C_{\delta_2})-g(C_{\delta_2})=2(M_{\delta_2}-M'_{\delta_2})-1,$ which is again odd and therefore non-zero.  Furthermore, since the sum of the values of the clouds around $\widetilde{V}_t$ must equal $n-1$ in each cell, we have 
\[
f(C_0)-g(C_0)+f(C_1)-g(C_1)+f(C_2)-g(C_2)=0,
\]
so  that $g(C_\delta)-f(C_\delta)>0$ and $g(C_{\delta'})-f(C_{\delta'})<0$ for two directions $\delta\ne\delta',$ so again at least one cloud is categorized as positive and one as negative, so $[c_i]\ne[d_j].$

Therefore, in all cases, the conditions of Lemma \ref{lemma:products} are met, so 
\[
TC(UD^n(\Gamma)),TC(D^n(\Gamma))\ge2q+1.
\]
This, combined with the upper bounds stated at the beginning of the proof, gives the result.
\end{proof}

It is worth noting that if we insist no vertices have degree 3, the statement of Theorem \ref{thm:Main} becomes much simpler and determines the topological complexity for all $n$ for both configuration spaces, provided they are connected. 
\begin{cor}\label{cor:NoDegreeThree}
Let $\Gamma$ be a tree with no vertices of degree 3. Let $k=\min\left\{\left\lfloor\frac{n}{2}\right\rfloor,m(\Gamma)\right\}.$  Then, $TC(UC^n(\Gamma))= 2k+1.$  Also, if $n=1$ or $m(\Gamma)\ge1,$ then $TC(C^n(\Gamma))=2k+1.$
\end{cor}

\begin{proof}  Let $r$ and $s$ be as in Theorem \ref{thm:Main}, so that $r=m(\Gamma)=:m,$ and $s=0.$  If $m\ge1$ and $n\ge2m,$ the claim follows from statement \ref{nLarge} of the theorem.  If $n=2q+\epsilon<2m=2r,$ with $\epsilon\in\{0,1\}$ and $q\ge1,$ then $r>q,$ so $2(q-r)<0=s,$ and the claim follows from statement \ref{nSmall}.  If $m=0,$ so that $\Gamma$ is homeomorphic to a closed interval, then Lemma \ref{lemma:MotionPlanning} gives $TC(UD^n(\Gamma))\le1,$ but $TC(X)\ge1$ for any space, so, $TC(UC^n(\Gamma))=1.$ For the remaining case, $n=1,\ m\ge1,$ we have $C^n(\Gamma)=UC^n(\Gamma)=\Gamma,$ so all three spaces have topological complexity 1 since $\Gamma$ is contractible.
\end{proof}

Now, we discuss how in some sense, the results in Theorem \ref{thm:Main} are the best we can achieve with the methods used here.  Consider the case $2m\le n<2m+k,$ with $m$ and $k$ as in the first part of the theorem.  In this case, a construction similar to the one given in the proof shows that there is a critical $m$-cell, which corresponds to a non-zero $m$-dimensional cohomology class, so that the space $UD^n(\Gamma)$ cannot be homotopic to a space of dimension less than $m,$ so the dimensional bound given in Theorem \ref{thm:upperbound} cannot improve the bound given by the explicit motion planning algorithm in Lemma \ref{lemma:MotionPlanning}.  Likewise, if $q,r,$ and $s$ are as in the second part of the theorem, but the appropriate collection of arcs does not exist, then there will still be a critical $q$-cell, so again the dimension cannot improve the upper bound.  The following shows that Lemma \ref{lemma:products} cannot be used to get improved lower bounds.
\begin{prop}\label{prop:nobetter}
Let $\Gamma$ be a tree with $m:=m(\Gamma)\ge1.$
\begin{enumerate}
\item\label{nobetter:nLarge} Let $k$ be as in statement \ref{nLarge} of Theorem \ref{thm:Main}, and assume there is as least one vertex of degree 3 so that $k\ge1.$  Let $n$ satisfy $n\ge 2m,$ and consider any two critical $m$-cells $\Phi$ and $\Psi$ of $UD^n(\Gamma).$ If $\{[c_1],\dots,[c_m]\}$ and $\{[d_1],\dots,[d_m]\}$ are the unique collections of equivalence classes of 1-cells having $[\Phi]$ and $[\Psi]$ as their least upper bounds, and for all $i$ and $j,$ we have $[c_i]\ne[d_j],$ then $n\ge2m+k.$
\item\label{nobetter:nNotLarge} Let $q,r,s,\epsilon$ be as in statement \ref{nMedium} of Theorem \ref{thm:Main}, and consider critical $q$-cells $\Phi$ and $\Psi$ of $UD^n(\Gamma).$  If $\{[c_1],\dots,[c_q]\}$ and $\{[d_1],\dots,[d_q]\}$ are the unique collections of equivalence classes of 1-cells having $[\Phi]$ and $[\Psi]$ as their least upper bounds and for all $i$ and $j,$ we have $[c_i]\ne[d_j],$ then 
\begin{enumerate}
\item\label{nobetterEven} if $\epsilon=0,$ there is some $k\ge1$ such that there exist oriented arcs $A_1,\dots,A_k$ with the following properties:
\begin{enumerate}
\item The endpoints of each $A_l$ are (distinct) essential vertices, neither of which is an endpoint of any other $A_{l'},$
\item  There are $r'\le r$ vertices of degree greater than 3 which are not the endpoints of any $A_l,$
\item  There is a collection $\mathcal{V}$ of degree-3 vertices, with $|\mathcal{V}|\ge q-r'-k$ such that $\{A_i\}_{i=1}^k$ is allowable for $\mathcal{V}.$
\end{enumerate}
\item\label{nobetterOdd} if $\epsilon=1,$ there is an arc $A_0$ whose endpoints have no restrictions and whose interior includes a collection $\mathcal{W}'$ of $s'\le q$ distinct vertices of degree 3, and if $s'<q-r,$ there are arcs $A_1,\dots,A_k,$ as above whose endpoints are also not vertices in $\mathcal{W}',$ and there is another collection of degree-3 vertices, $\mathcal{W}$, such that $\mathcal{W}\cap\mathcal{W}'=\emptyset,\ |\mathcal{W}|\ge q-r'-k-s'$ and $\{A_i\}_{i=1}^k$ is allowable for $\mathcal{W},$ where $r'$ is as above.
\end{enumerate} 
\end{enumerate}
\end{prop}

\begin{proof} We use the contrapositive for both statements \ref{nobetter:nLarge} and \ref{nobetter:nNotLarge}.  For the first statement, let $\Phi$ and $\Psi$ be critical $m$-cells, and assume $2m\le n<2m+k.$  For each essential vertex $v,$ each cell must contain exactly one edge $e$ having $\iota(e)=v$ and a blocked vertex $u$ which makes $e$ order-disrespecting. There is no choice for $e$ or $u$ if $v$ is of degree 3.  

If $n>2m,$ let $w_{1},\dots,w_{n-2m}$ be the remaining vertices in $\Phi,$ and let $w'_1,\dots,w'_{n-2m}$ be the remaining vertices in $\Psi.$  For $l=1,\dots,n-2m<k,$ let $A_l$ be the geodesic from $w_l$ to $w'_l,$ oriented so that $w_l$ is the initial endpoint and $w'_l$ is the terminal endpoint if $w_l\ne w'_l.$  If $w_l=w'_l,$ extend $A_l$ slightly so that it is a small arc starting at $w_l$ which doesn't intersect any essential vertices.  This gives a collection of less than $k$ oriented arcs in $\Gamma,$ so by the minimality of $k,$ this collection cannot be allowable for the set of vertices of degree 3 in $\Gamma.$  It is not possible that a degree-3 vertex $v$ is an endpoint of any $A_l,$ since no $w_l$ or $w'_l$ can be essential.  Therefore, there is some degree-3 vertex $v$ which has the property $\eta_0(v)=\eta_1(v)=\eta_2(v)=0.$  Let $e$ be the edge in direction 2 from $v.$  This edge must be in both $\Phi$ and $\Psi.$  Suppose $i$ and $j$ have the property that $e$ is the unique edge in $[c_i]$ and $[d_j].$  Then, in the notation from the proof of Theorem \ref{thm:Main}, for $\delta=0,1,2$ and each cloud $C_\delta$ in the diagrams for $[c_i]$ and $[d_j],$ we have $f(C_\delta)=2m_\delta+W_\delta+\lambda_\delta$ and $g(C_\delta)=2m_\delta+W'_\delta+\lambda_\delta.$  But, as above, we have $W_\delta-W'_\delta=\eta_\delta(v)=0$ for each $\delta,$ and therefore $W_\delta=W'_\delta$ and $f(C_\delta)=g(C_\delta),$ for each $\delta,$ so $[c_i]=[d_j].$

If $n=2m,$ then for each vertex $v$ of degree 3, the edge $e$ in direction 2 from $v$ must appear in both cells, so again if $[c_i]$ and $[d_j]$ are the equivalence classes which contain the edge $e,$ we now have $f(C_\delta)=2m_\delta+\lambda_\delta=g(C_\delta)$ for each $\delta=0,1,2,$ so $[c_i]=[d_j].$

For statement \ref{nobetter:nNotLarge}, assume that the collections of arcs and vertices in \ref{nobetterEven} or \ref{nobetterOdd} do not exist for the appropriate value of $\epsilon.$  Note first that if $\epsilon=0,$ any critical $q$-cell must consist solely of edges $e_1,\dots,e_q$ and blocked vertices $u_1,\dots,u_q$ which force the edges $e_1,\dots,e_q$ to be order-disrespecting.  In particular, each $e_i$ has the property that $\iota(e_i)$ is an essential vertex.  If $\epsilon=1,$ then the same is true, except that the cell contains one additional blocked vertex.  Let $\Phi$ and $\Psi$ be critical $q$-cells, and let $\mathcal{S}$  (resp. $\mathcal{S}'$) be the set of essential vertices $v$ such that $\iota(e)=v$ for some $e$ in $\Phi$ (resp. $\Psi$), so $|\mathcal{S}|=|\mathcal{S}'|=q.$  

Let $\mathcal{T}=\mathcal{S}\setminus (\mathcal{S}\cap \mathcal{S}')$ and $\mathcal{T}'=\mathcal{S}'\setminus(\mathcal{S}\cap \mathcal{S}'),$ so $|\mathcal{T}|=|\mathcal{T}'|=:k.$ The fact that $s<2(q-r)$ implies that at least one vertex of degree 3 appears in $\mathcal{S}\cap \mathcal{S}'.$  Let $\widetilde{V}_1,\dots,\widetilde{V}_{\overline{s}}$ be the vertices of degree 3 in $\mathcal{S}\cap \mathcal{S}',$ so that $\overline{s}>0$ and let $\widehat{V}_1,\dots,\widehat{V}_{\overline{r}}$  be the vertices of degree greater than 3 in $\mathcal{S}\cap \mathcal{S}'.$  The edge in direction 2 from each vertex $\widetilde{V}_t$ must be included in both $\Phi$ and $\Psi.$  In what follows, $[c_i]$ and $[d_j]$ are the equivalence classes which contain the edge in direction 2 from whichever vertex $\widetilde{V}_t$ is being discussed.

If $\epsilon=1,$ let $x$ (resp. $x'$) be the additional vertex in $\Phi$ (resp. $\Psi$).  Let $A_0$ be the geodesic from $x$ to $x'$ and extend $A_0$ slightly if $x= x'$ so that it is a small arc which doesn't intersect any vertex in $\mathcal{S}$ or $\mathcal{S'}.$  Let $\mathcal{W}'$ be the set of degree-3 vertices in $\mathcal{S}\cap\mathcal{S}'$ which fall on the interior of $A_0,$ and let $s'=|\mathcal{W}'|.$  By the assumption of the non-existence of the appropriate collections of arcs and vertices, we have $s'<q-r.$

First consider the case $k=0.$ Here, we have $\mathcal{S}=\mathcal{S}',$ and $\overline{s}+\overline{r}=q.$  If $\epsilon=0,$ then for each vertex $\widetilde{V}_t\in \mathcal{S}\cap \mathcal{S}'=\mathcal{S}=\mathcal{S'},$ we have, with the notation from the proof of Theorem \ref{thm:Main}, $f(C_\delta)=2(\widetilde{M}_\delta+\widehat{M}_\delta)+\lambda_\delta=g(C_\delta),$ so $[c_i]=[d_j].$  If $\epsilon=1,$ since we have $s'<q-r,$ there must be at least one vertex $\widetilde{V}_t$ which does not fall on $A_0.$  Then, $x$ and $x'$ must fall in the same cloud in the system of clouds determined by the edge in direction 2 from $\widetilde{V}_t.$  Again using the notation from the proof of Theorem \ref{thm:Main}, for each $\delta,$ we have $\mu_\delta=\mu'_\delta,$ so $f(C_\delta)=2(\widetilde{M}_\delta+\widehat{M}_\delta)+\lambda_\delta+\mu_\delta=2(\widetilde{M}_\delta+\widehat{M}_\delta)+\lambda_\delta+\mu'_\delta=g(C_\delta),$ and therefore $[c_i]=[d_j].$

If $k\ge1,$ let $V_1,\dots,V_{k}$ and $V'_1,\dots,V'_{k}$ be the vertices in $\mathcal{T}$ and $\mathcal{T}',$ respectively, and let $A_l$ be the oriented geodesic from $V_l$ to $V'_l.$  This gives a collection of $k$ arcs with $2k$ distinct essential endpoints if $\epsilon=0,$ and one additional arc $A_0$ if $\epsilon=1.$  Also note that $\overline{s}+\overline{r}+k=q$ and $\overline{r}\le r',$ where $r'$ is the number of vertices of degree greater than 3 which are not endpoints of any $A_l,$ so $\overline{s}=q-\overline{r}-k\ge q-r'-k.$  

For $\epsilon=0,$ let $\mathcal{V}=\{\widetilde{V}_1,\dots,\widetilde{V}_{\overline{s}}\},$ so that $|\mathcal{V}|=\overline{s}\ge q-r'-k.$  By assumption, the arcs $\{A_i\}_{i=1}^k$ cannot be allowable for $\mathcal{V}.$  Any vertex $v\in \mathcal{V}$ cannot be the endpoint of any $A_l$ since no endpoint is in $\mathcal{S}\cap \mathcal{S}'.$   So, there is at least one degree-3 vertex $\widetilde{V}_t$ such that $\eta_0(\widetilde{V}_t)=\eta_1(\widetilde{V}_t)=\eta_2(\widetilde{V}_t)=0.$  So, we have $f(C_\delta)-g(C_\delta)=2(M_\delta-M'_\delta)=2\eta_\delta(v)=0,$ for each $\delta\in\{0,1,2\},$ and therefore again each cloud has the same value in the diagram for $[c_i]$ as it does in the diagram for $[d_j].$  So, $[c_i]=[d_j].$

For $\epsilon=1,$ let $\mathcal{W}$ be the set of degree-3 vertices in $(\mathcal{S}\cap \mathcal{S}')\setminus\mathcal{W}',$ so that $\mathcal{W}\cap\mathcal{W}'=\emptyset$ and $|\mathcal{W}|=\overline{s}-s'\ge q-r'-k-s',$ and therefore, again by assumption, $\{A_i\}_{i=1}^k$ is not allowable for $\mathcal{W}.$  No vertex in $\mathcal{W}$ can be an endpoint of any arc $A_l$, so there must be some vertex $\widetilde{V}_t\in\mathcal{W}$ with $\eta_\delta(\widetilde{V}_t)=0$ for $\delta=0,1,2.$  Now, we have $f(C_\delta)-g(C_\delta)=2(M_\delta-M'_\delta)+\mu_\delta-\mu'_\delta,$ but since $\widetilde{V}_t$ is not on $A_0,$ we must have both $x$ and $x'$ in the same cloud in the system of clouds determined by the edge in direction 2 from $\widetilde{V}_t,$ so $\mu_\delta=\mu'_\delta,$ and therefore $f(C_\delta)-g(C_\delta)=2(M_\delta-M'_\delta)=2\eta_\delta(\widetilde{V}_t)=0$ for each direction $\delta,$ so again we have $[c_i]=[d_j].$
\end{proof}

To give a better idea of the values of $n$ for which Theorem \ref{thm:Main} determines $TC(UC^n(\Gamma))$ (and $TC(C^n(\Gamma))$), we consider all values of $n\ge2$ with $\Gamma$ as in the left of Figure \ref{fig:PartOneA}, where we have $m=11,\ r=2,$ and $s=9.$  If $n$ is either sufficiently large or sufficiently small, it is easy to determine if the theorem applies.  For $n\ge 25,$ statement \ref{nLarge} of the theorem applies, and for $2\le n\le 13,$ statement \ref{nSmall} applies.  For $22\le n\le24,$ the theorem does not apply.  For the ``middle" values of $n$, namely $14\le n\le21,$ only statement \ref{nMedium} might apply.

For $n=14$ statement \ref{nMedium} does apply, since we may chose the arcs $A_1,\ A_2,$ and $A_3$ as in Figure \ref{fig:PartTwoBi}, and let $\mathcal{V}$ consist of the vertices $\widetilde{V}_1,\ \widetilde{V}_2,$ and $\widetilde{V}_3$ (as labeled in Figure \ref{fig:PartTwoBi}).  Then, we have $k=3,\ r'=2, |\mathcal{V}|=3\ge q-r'-k=7-2-3=2,$ and the arcs $A_1,\ A_2,$ and $A_3$ are again allowable for $\mathcal{V}.$  Note the vertex $\widetilde{V}_3$ would not be used in the construction of the cells $\Phi$ and $\Psi$ for $n=14.$  For $n=15,$ we can again let $A_0$ be the unique edge which has $\ast$ as one of its endpoints, and keep $A_1,\ A_2,$ and $A_3$ the same, so the theorem also applies for $n=15.$  The cases $n=16,\ n=17,$ and $n=19$ are covered in the proof.  

For $n=18,$ suppose the appropriate collection of arcs $A_1,\dots,A_k$ exists and is allowable for a set $\mathcal{V}$ of vertices of degree 3, so that we have $|\mathcal{V}|\ge q-r'-k.$  On the other hand, there must be a total of $2k$ distinct vertices which are the endpoints of the arcs.  No vertex in $\mathcal{V}$ can be any of these endpoints, so we must have $|\mathcal{V}|\le s-2k+r-r',$ since there are $r-r'$ vertices of degree greater than 3 which are the endpoint of some arc, and the remaining $2k-(r-r')$ endpoints must be of degree 3.  So, we have 
\[
9-r'-k=q-r'-k\le|\mathcal{V}|\le s-2k+r-r'=9-2k+2-r'.
\]
Comparing the left and right sides, we see that we must have $k\le2.$  Since we also must have $r'\le r=2,$ this gives $|\mathcal{V}|\ge 9-2-2=5.$  But, the only vertices of degree 3 which can fall on the interior of any arc with essential endpoints are the vertices labeled $v_4,\ v_5,$ and $v_8$ in Figure \ref{fig:PartOneA}. so $|\mathcal{V}|\le3,$ arriving at a contradiction, so the theorem does not apply for $n=18,$ and for similar reasons, the theorem does not apply for $n=20.$

For $n=21,$ suppose the appropriate collection of arcs $A_0,A_1,\dots,A_k$ and collections of vertices $\mathcal{W}$ and $\mathcal{W}'$ exist, so that $|\mathcal{W}\cup\mathcal{W}'|\ge q-r'-k=10-r'-k.$  Similar to above, since no vertex in $\mathcal{W}\cup\mathcal{W'}$ can be an endpoint of any arc $A_1,\dots,A_k,$ we must have 
\[
10-r'-k=q-r'-k\le|\mathcal{W}\cup\mathcal{W}'|\le s-2k+r-r'=9-2k+2-r',
\]
so that $k\le1.$  Since again we must have $r'\le2,$ this gives $|\mathcal{W}\cup\mathcal{W}'|\ge10-r'-k\ge10-2-1=7.$  However, again the only vertices of degree 3 which can fall on the interior of any arc $A_1$ are the vertices labeled $v_4,\ v_5,$ and $v_8$ in Figure \ref{fig:PartOneA}, and here it is clear that at most two additional vertices could be included in the interior of the arc $A_0,$ so that $|\mathcal{W}\cup\mathcal{W}'|\le5,$ a contradiction, so the theorem does not apply for $n=21.$

So, the only values of $n\ge2$ for which Theorem \ref{thm:Main} does not determine $TC(UC^n(\Gamma))$ and $TC(C^n(\Gamma))$ are $n=18$ and $20\le n\le24.$  

\section*{Acknowledgements} This project was supported by a research fellowship awarded by Lehigh University which was made possible thanks to a generous donation by Dale Strohl, class of 1958.  The author would also like to thank his advisor, Donald Davis, for his suggestion of the topic and his help in preparing this article, and the referee of this article for carefully reading the article and pointing out a mistake in an earlier version.  This correction led to the notion of allowable collections of arcs.  The author also extends his gratitude to Daniel Farley for his discussions regarding the proof of Lemma \ref{lemma:products}.

\bibliography{References}
\bibliographystyle{gtart}
\end{document}

%% file: Labeling/LabeledFigures.tex
\def\VertexOrdering{
\begin{figure}[h]
\labellist
\tiny\hair 2pt
\pinlabel $\ast$ at 181 36
\pinlabel $1$ at 189 95
\pinlabel $2$ at 181 162
\pinlabel $3$ at 150 153
\pinlabel $4$ at 117 173
\pinlabel $5$ at 91 188
\pinlabel $6$ at 223 161
\pinlabel $7$ at 274 190
\pinlabel $8$ at 241 231
\pinlabel $9$ at 223 261
\pinlabel $10$ at 206 290
\pinlabel $11$ at 309 249
\pinlabel $12$ at 332 288
\pinlabel $13$ at 286 348
\pinlabel $14$ at 270 386
\pinlabel $15$ at 330 400
\pinlabel $16$ at 345 321
\pinlabel $17$ at 359 354
\pinlabel $18$ at 369 384
\endlabellist
\centering
\includegraphics[height=4.5cm]{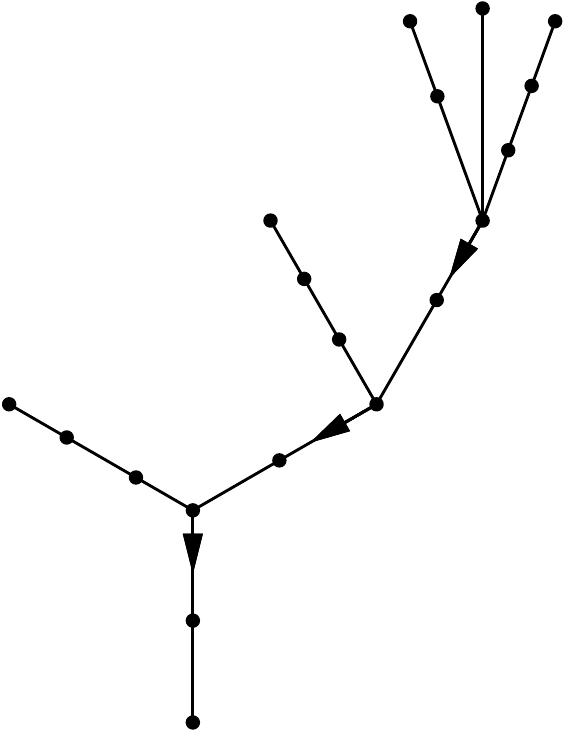}
\caption{The ordering of the vertices on a tree $\Gamma$}
\label{fig:VertexOrdering}
\end{figure}
}

\def\AllowableOrientations{
\begin{figure}[h]
\labellist
\tiny\hair 2pt
\pinlabel $v$ at 211 227
\pinlabel $v$ at 593 227
\pinlabel $\ast$ at 175 15
\pinlabel $\ast$ at 557 15
\endlabellist
\centering
\includegraphics[height=4.5cm]{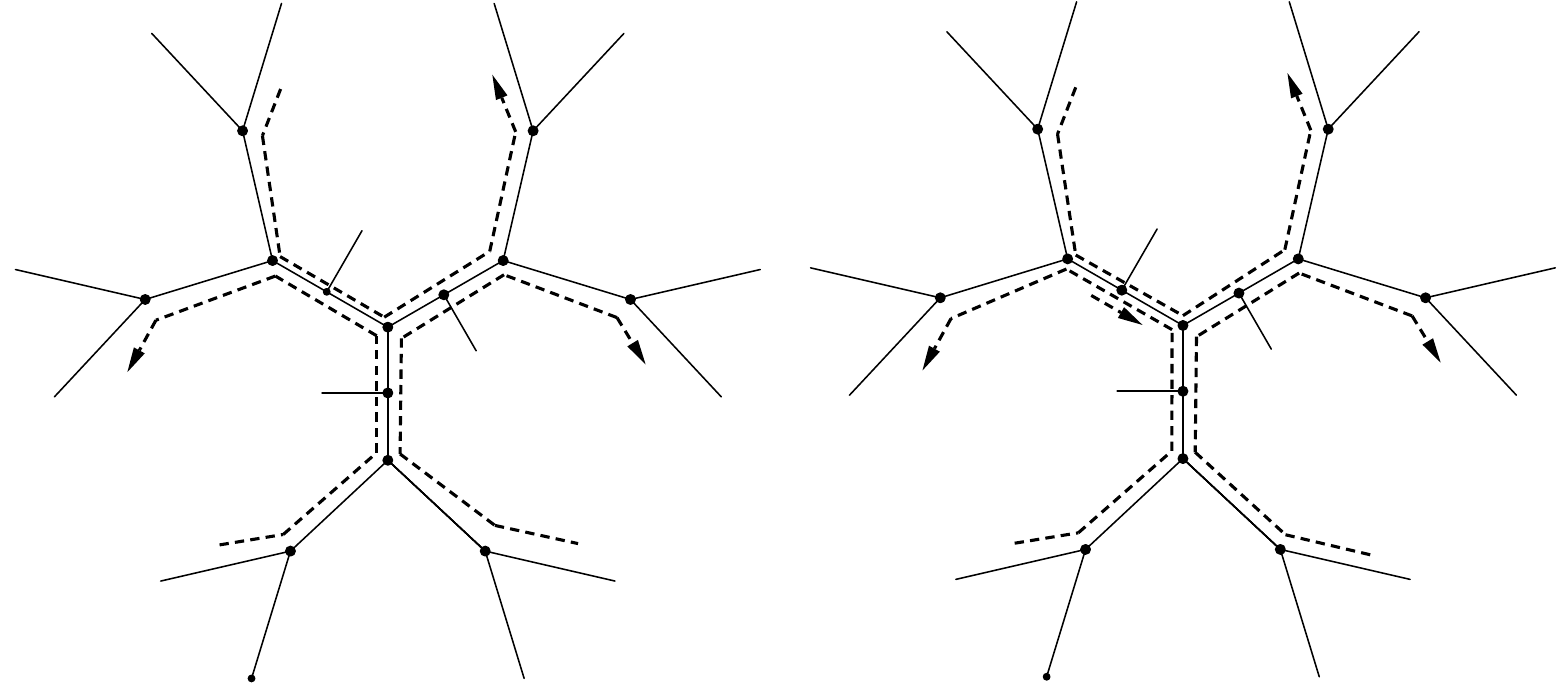}
\caption{Two collections of arcs in a graph $\Gamma$}
\label{fig:AllowableOrientations}
\end{figure}
}

\def\Cells{
\begin{figure}[h]
\labellist
\tiny\hair 2pt
\pinlabel $e$ at 192 133
\pinlabel $3$ at 132 130
\pinlabel $10$ at 550 272
\pinlabel $9$ at 566 243
\pinlabel $f$ at 940 300
\pinlabel $16$ at 1000 300
\endlabellist
\centering
\includegraphics[height=4.5cm]{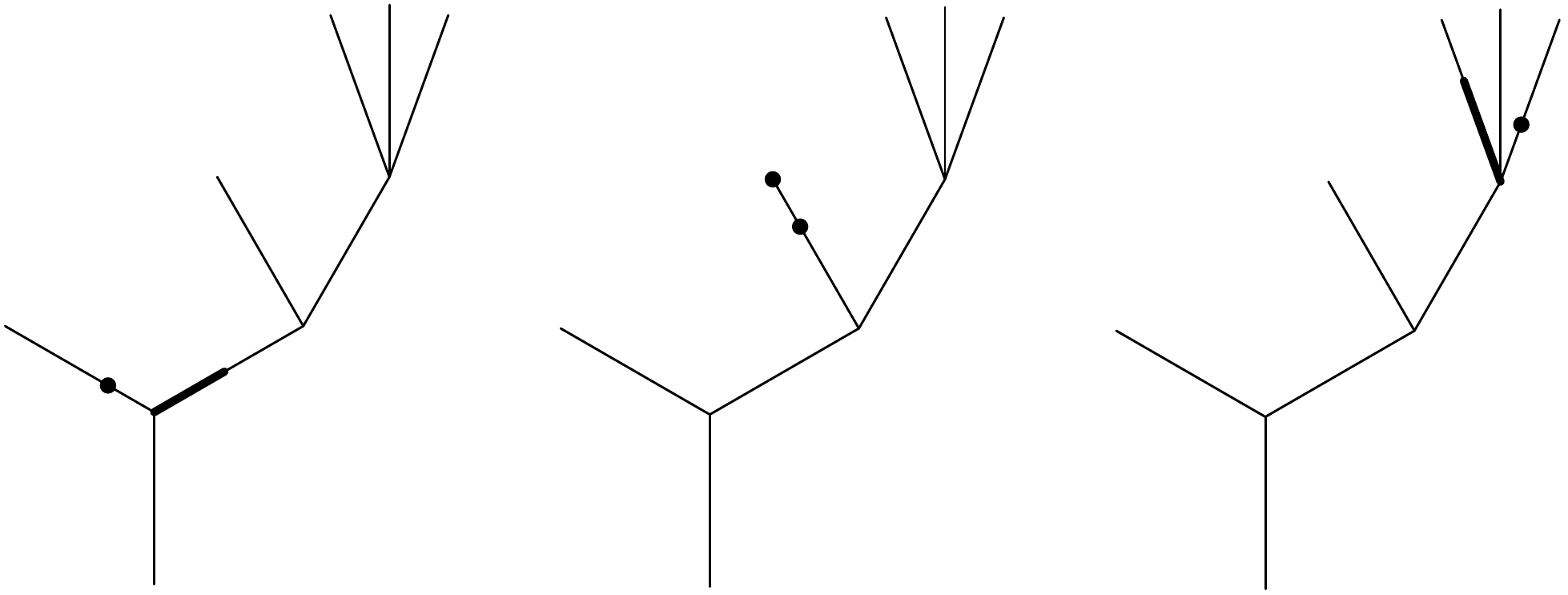}
\caption{Three different cells in $UD^2(\Gamma)$}
\label{fig:Cells}
\end{figure}
}

\def\CriticalCells{
\begin{figure}[h]
\labellist
\tiny\hair 2pt
\pinlabel $\ast$ at 125 20
\endlabellist
\centering
\includegraphics[height=4.5cm]{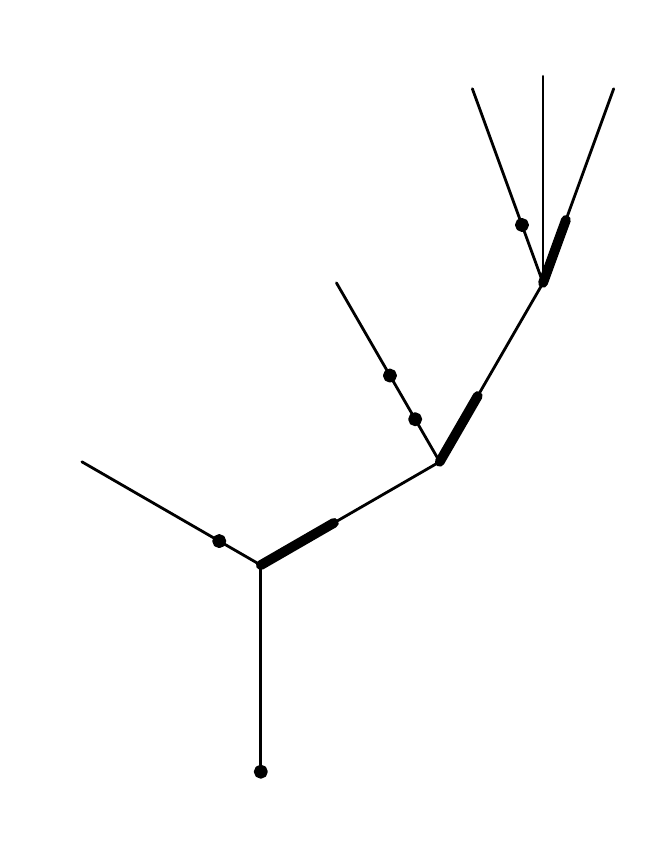}
\caption{A critical 3-cell $c$ in $UD^8(\Gamma)$}
\label{fig:CriticalCells}
\end{figure}
}

\def\Clouds{
\begin{figure}[h]
\labellist
\tiny\hair 2pt
\pinlabel $1$ at 168 82
\pinlabel $1$ at 105 170
\pinlabel $0$ at 233 160
\pinlabel $2$ at 221 244
\pinlabel $0$ at 291 240
\pinlabel $1$ at 262 330
\pinlabel $0$ at 308 348
\pinlabel $0$ at 333 333
\pinlabel $1$ at 498 82
\pinlabel $1$ at 435 170
\pinlabel $0$ at 563 160
\pinlabel $2$ at 551 244
\pinlabel $2$ at 621 240
\pinlabel $0$ at 828 82
\pinlabel $1$ at 767 170
\pinlabel $0$ at 897 160
\pinlabel $2$ at 885 244
\pinlabel $3$ at 955 240
\endlabellist
\centering
\includegraphics[height=4.5cm]{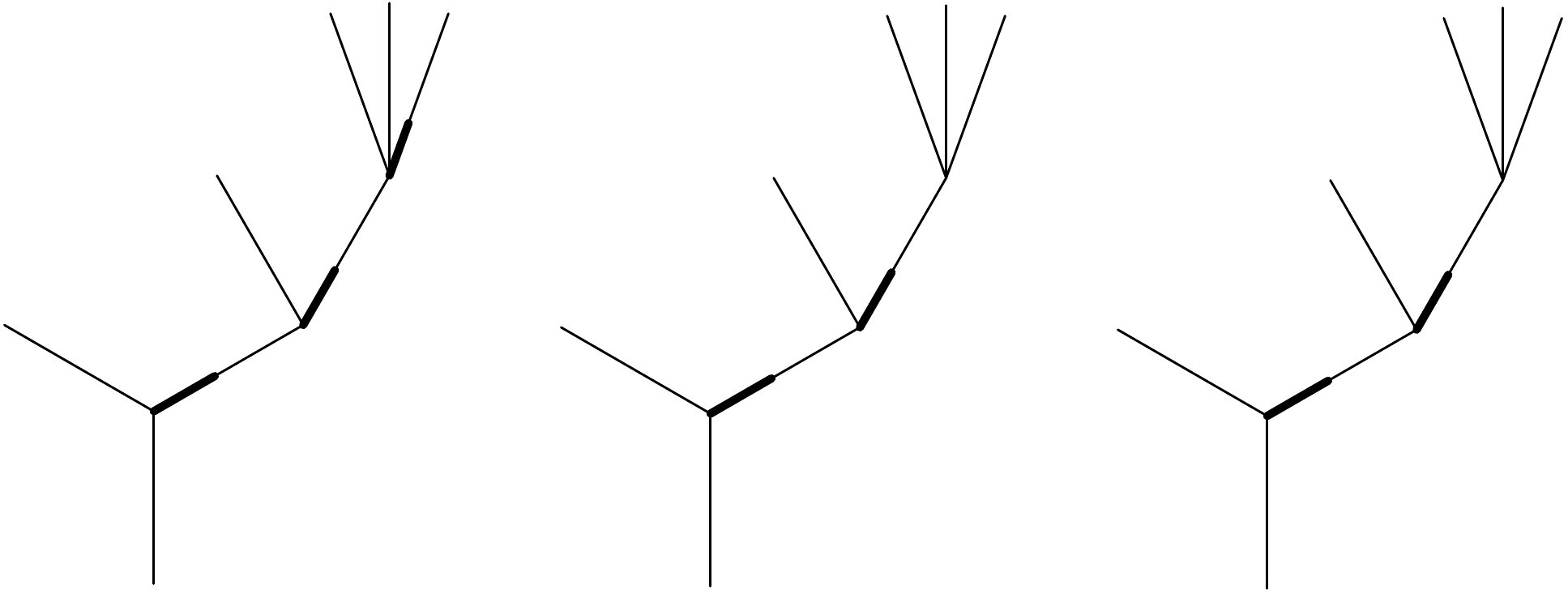}
\caption{Cloud diagrams for classes $[c],$ with $c$ as in Figure \ref{fig:CriticalCells} (left), $[d]$ (middle), and $[d']$ (right)}
\label{fig:Clouds}
\end{figure}
}

\def\PartOneA{
\begin{figure}[h]
\labellist
\tiny\hair 2pt
\pinlabel $\ast$ at 140 36
\pinlabel$A_1$ at 170 84 
\pinlabel $A_2$ at 196 201
\pinlabel $A_3$ at 305 240
\pinlabel $v_1$ at 139 127
\pinlabel $v_2$ at 173 147
\pinlabel $v_3$ at 154 209
\pinlabel $v_4$ at 257 175
\pinlabel $v_5$ at 244 253
\pinlabel $v_6$ at 206 287
\pinlabel $v_7$ at 308 280
\pinlabel $v_8$ at 283 159
\pinlabel $v_9$ at 315 129
\pinlabel $v_{10}$ at 356 133
\pinlabel $v_{11}$ at 338 70
\pinlabel $u_1$ at 494 131
\pinlabel $e_1$ at 524 104
\pinlabel $u_2$ at 536 154
\pinlabel $e_2$ at 576 134
\pinlabel $u_3$ at 508 176
\pinlabel $e_3$ at 546 199
\pinlabel $u_4$ at 598 186
\pinlabel $e_4$ at 625 171
\pinlabel $u_5$ at 581 243
\pinlabel $e_5$ at 637 237
\pinlabel $u_6$ at 536 265
\pinlabel $e_6$ at 572 290
\pinlabel $u_7$ at 671 290
\pinlabel $e_7$ at 668 256
\pinlabel $u_8$ at 653 133
\pinlabel $e_8$ at 628 142
\pinlabel $u_9$ at 684 162
\pinlabel $e_9$ at 684 128
\pinlabel $u_{10}$ at 715 161
\pinlabel $e_{10}$ at 724 126
\pinlabel $u_{11}$ at 706 69
\pinlabel $e_{11}$ at 668 78
\pinlabel $\ast=w_1$ at 503 33
\pinlabel $w_2$ at 564 228
\pinlabel $w_3$ at 694 242
\pinlabel $w_4$ at 522 60
\pinlabel $w_5$ at 522 75
\pinlabel $u'_1$ at 860 131
\pinlabel $e'_1$ at 890 104
\pinlabel $u'_2$ at 934 161
\pinlabel $e'_2$ at 917 120
\pinlabel $u'_3$ at 874 176
\pinlabel $e'_3$ at 912 199
\pinlabel $u'_4$ at 964 187
\pinlabel $e'_4$ at 990 174
\pinlabel $u'_5$ at 947 239
\pinlabel $e'_5$ at 1003 237
\pinlabel $u'_6$ at 904 260
\pinlabel $e'_6$ at 938 290
\pinlabel $u'_7$ at 1037 290
\pinlabel $e'_7$ at 1034 250
\pinlabel $u'_8$ at 1019 131
\pinlabel $e'_8$ at 992 142
\pinlabel $u'_9$ at 1050 156
\pinlabel $e'_9$ at 1050 124
\pinlabel $u'_{10}$ at 1081 165
\pinlabel $e'_{10}$ at 1090 126
\pinlabel $u'_{11}$ at 1072 96
\pinlabel $e'_{11}$ at 1030 80
\pinlabel $w'_1$ at 1010 48
\pinlabel $w'_2$ at 878 265
\pinlabel $w'_3$ at 1124 160
\pinlabel $\ast=w'_4$ at 869 33
\pinlabel $w'_5$ at 888 60
\endlabellist
\centering
\includegraphics[height=4cm]{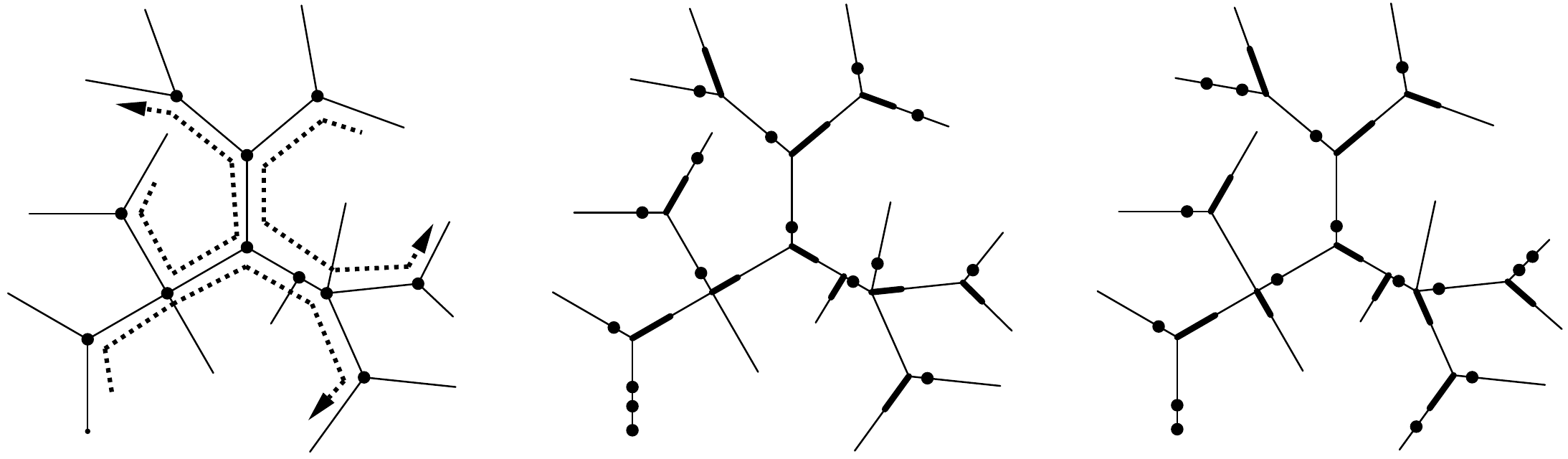}
\caption{A graph $\Gamma$ (left) and the critical cells $\Phi$ and $\Psi$ in $UD^{27}(\Gamma)$ (center and right)}
\label{fig:PartOneA}
\end{figure}
}

\def\PartOneB{
\begin{figure}[h]
\labellist
\scriptsize\hair 2pt
\pinlabel $10$ at 252 234
\pinlabel $8$ at 325 307
\pinlabel $8$ at 367 163
\pinlabel $8$ at 815 234
\pinlabel $8$ at 892 307
\pinlabel $10$ at 936 163
\endlabellist
\centering
\includegraphics[height=4cm]{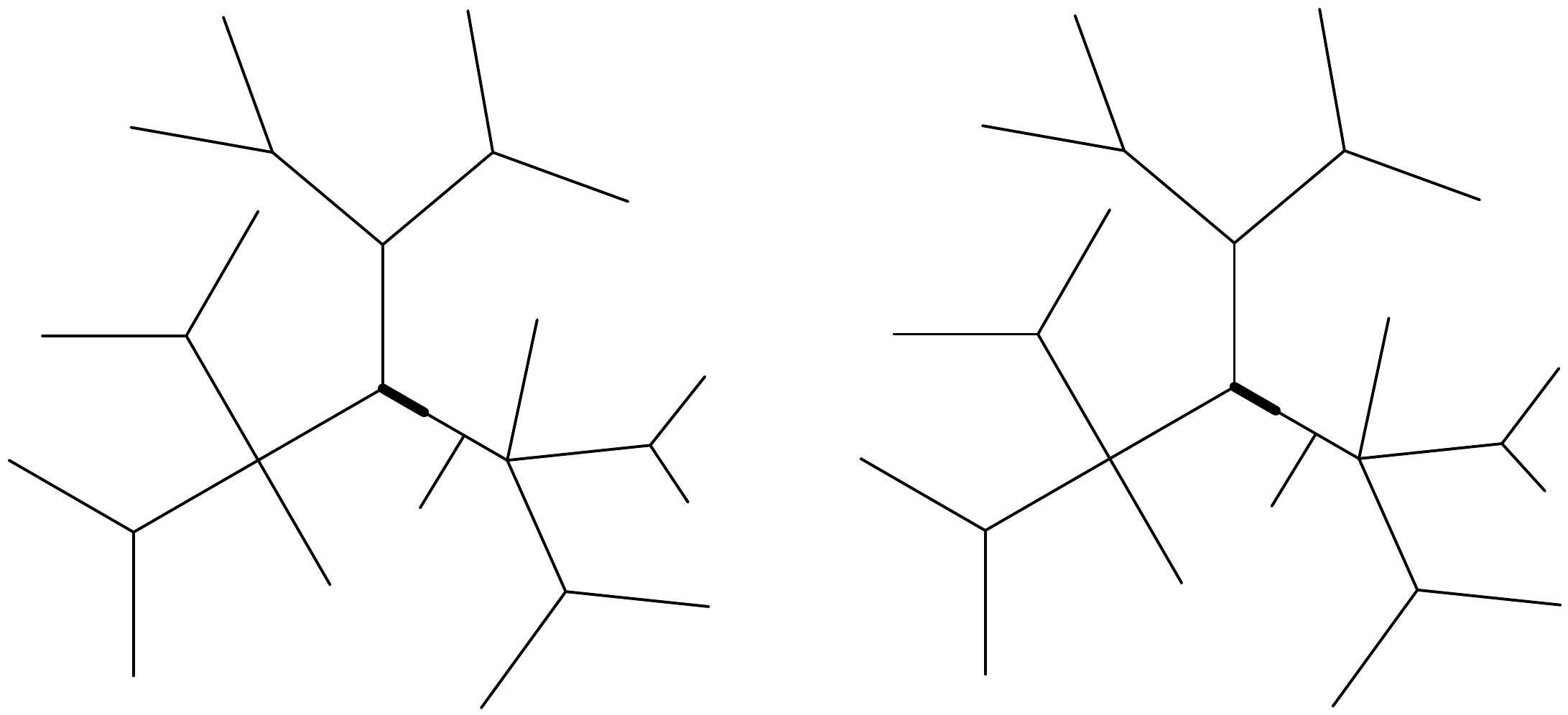}
\caption{Cloud representations for $[c_4]$ (left) and $[d_4]$ (right) corresponding to the cells $\Phi$ and $\Psi$ in $UD^{27}(\Gamma)$}
\label{fig:PartOneB}
\end{figure}
}

\def\PartTwoA{
\begin{figure}[h!]
\labellist
\tiny\hair 2pt
\pinlabel $u_1$ at 190 217
\pinlabel $e_1$ at 260 191
\pinlabel $u_2$ at 416 228
\pinlabel $e_2$ at 422 181
\pinlabel $\widetilde{u}_1$ at 109 141
\pinlabel $\widetilde{e}_1$ at 180 142
\pinlabel $\widetilde{u}_2$ at 155 300
\pinlabel $\widetilde{e}_2$ at 212 295
\pinlabel $\widetilde{u}_3$ at 285 268
\pinlabel $\widetilde{e}_3$ at 315 216
\pinlabel $u'_1$ at 804 232
\pinlabel $e'_1$ at 787 162
\pinlabel $u'_2$ at 984 222
\pinlabel $e'_2$ at 944 172
\pinlabel $\widetilde{u}_4$ at 829 352
\pinlabel $\widetilde{e}_4$ at 906 336
\pinlabel $\widetilde{u}_5$ at 770 384
\pinlabel $\widetilde{e}_5$ at 809 431
\pinlabel $\widetilde{u}_6$ at 921 426
\pinlabel $\widetilde{e}_6$ at 958 376
\endlabellist
\centering
\includegraphics[height=3.7cm]{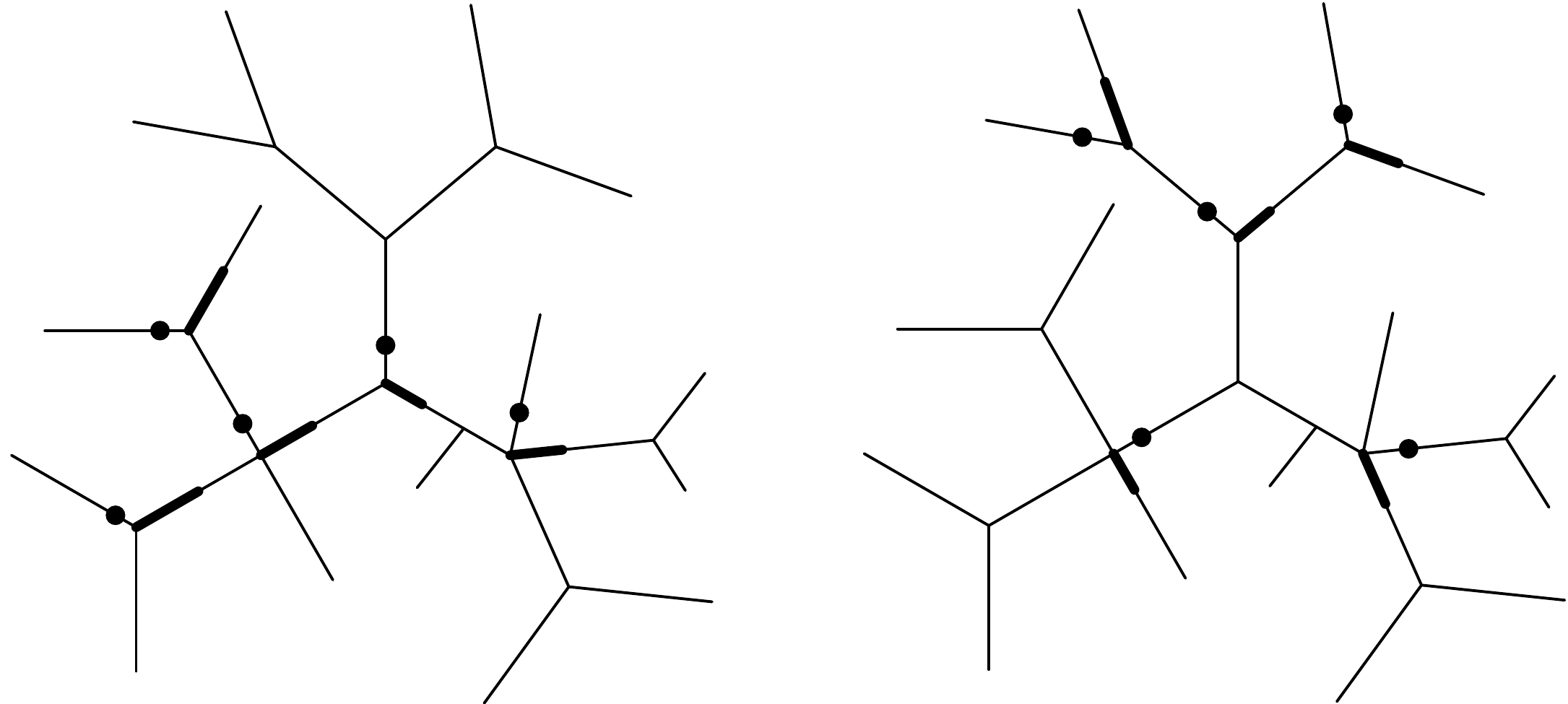}
\caption{The critical cells $\Phi$ and $\Psi$ in $UD^{10}(\Gamma)$}
\label{fig:PartTwoA}
\end{figure}
}

\def\PartTwoBi{
\begin{figure}[h]
\labellist
\tiny\hair 2pt
\pinlabel $A_1$ at 154 86
\pinlabel $A_2$ at 176 177
\pinlabel $A_3$ at 276 227
\pinlabel $V_1$ at 109 90
\pinlabel $V_2$ at 137 201
\pinlabel $V_3$ at 288 271
\pinlabel $V_1'$ at 316 55
\pinlabel $V_2'$ at 191 275
\pinlabel $V_3'$ at 365 135
\pinlabel $\widetilde{V}_1$ at 249 169
\pinlabel $\widetilde{V}_2$ at 228 239
\pinlabel $\widetilde{V}_3$ at 269 159
\pinlabel $\widehat{V}_1$ at 155 137
\pinlabel $\widehat{V}_2$ at 302 116
\pinlabel $U_1$ at 478 115
\pinlabel $E_1$ at 510 93
\pinlabel $U_2$ at 494 168
\pinlabel $E_2$ at 530 188
\pinlabel $U_3$ at 652 280
\pinlabel $E_3$ at 650 240
\pinlabel $\widetilde{U}_1$ at 579 177
\pinlabel $\widetilde{E}_1$ at 611 163
\pinlabel $\widetilde{U}_2$ at 564 230
\pinlabel $\widetilde{E}_2$ at 616 217
\pinlabel $\widetilde{U}_3$ at 639 118
\pinlabel $\widetilde{E}_3$ at 608 130
\pinlabel $\widehat{U}_1$ at 518 147
\pinlabel $\widehat{E}_1$ at 551 151
\pinlabel $\widehat{U}_2$ at 667 152
\pinlabel $\widehat{E}_2$ at 670 115
\pinlabel $U'_1$ at 1054 82
\pinlabel $E'_1$ at 1034 51
\pinlabel $U'_2$ at 888 250
\pinlabel $E'_2$ at 916 280
\pinlabel $U'_3$ at 1093 145
\pinlabel $E'_3$ at 1063 118
\pinlabel $\widetilde{U}_1$ at 942 177
\pinlabel $\widetilde{E}_1$ at 974 163
\pinlabel $\widetilde{U}_2$ at 930 230
\pinlabel $\widetilde{E}_2$ at 976 217
\pinlabel $\widetilde{U}_3$ at 1002 146
\pinlabel $\widetilde{E}_3$ at 968 130
\pinlabel $\widehat{U}_1'$ at 914 154
\pinlabel $\widehat{E}_1'$ at 898 108
\pinlabel $\widehat{U}_2'$ at 1028 145
\pinlabel $\widehat{E}_2'$ at 1005 106
\endlabellist
\centering
\includegraphics[height=4cm]{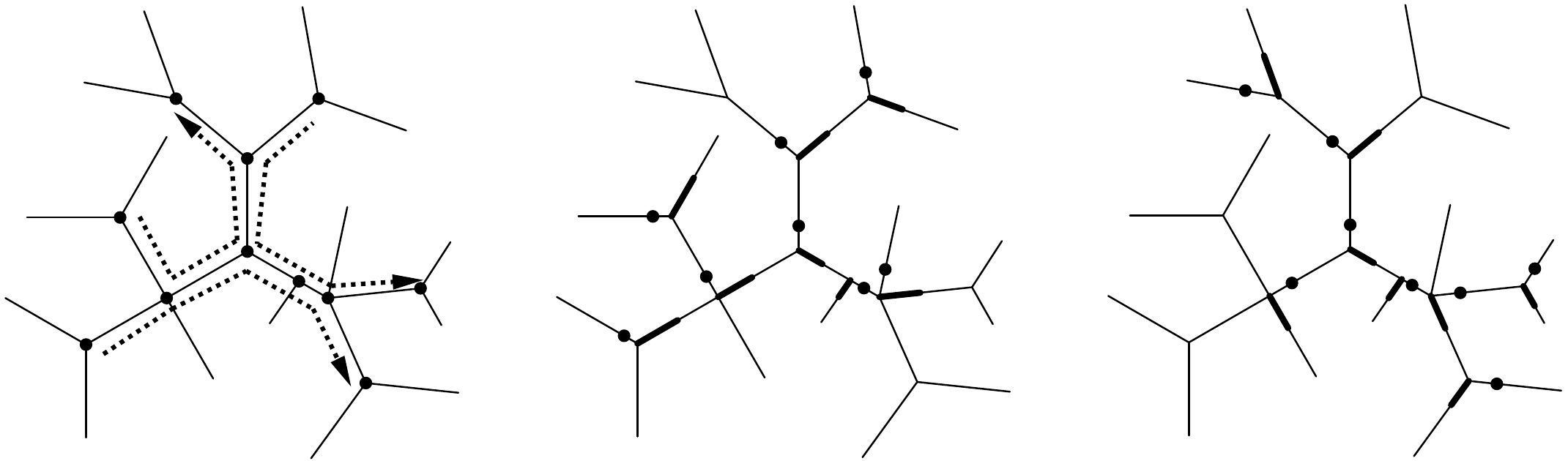}
\caption{The critical cells $\Phi$ and $\Psi$ in $UD^{16}(\Gamma)$ (middle and right)}
\label{fig:PartTwoBi}
\end{figure}
}

\def\PartTwoBii{
\begin{figure}[h]
\labellist
\tiny\hair 2pt
\pinlabel $A_1$ at 154 82
\pinlabel $A_2$ at 176 177
\pinlabel $A_0$ at 291 227
\pinlabel $V_1$ at 111 89
\pinlabel $V_2$ at 142 190
\pinlabel $V'_1$ at 313 50
\pinlabel $V'_2$ at 190 268
\pinlabel $\widetilde{V}_1$ at 252 168
\pinlabel $\widetilde{V}_2$ at 233 232
\pinlabel $\widetilde{V}_3$ at 288 266
\pinlabel $\widetilde{V}_4$ at 273 152
\pinlabel $\widetilde{V}_5$ at 361 129
\pinlabel $\widehat{V}_1$ at 163 130
\pinlabel $\widehat{V}_2$ at 302 111
\pinlabel $U_1$ at 479 112
\pinlabel $E_1$ at 507 87
\pinlabel $U_2$ at 499 162
\pinlabel $E_2$ at 532 187
\pinlabel $\widetilde{U}_1$ at 580 170
\pinlabel $\widetilde{E}_1$ at 612 158
\pinlabel $\widetilde{U}_2$ at 583 240
\pinlabel $\widetilde{E}_2$ at 616 214
\pinlabel $\widetilde{U}_3$ at 652 274
\pinlabel $\widetilde{E}_3$ at 652 236
\pinlabel $\widetilde{U}_4$ at 638 114
\pinlabel $\widetilde{E}_4$ at 611 130
\pinlabel $\widetilde{U}_5$ at 700 147
\pinlabel $\widetilde{E}_5$ at 727 121
\pinlabel $\widehat{U}_1$ at 519 138
\pinlabel $\widehat{E}_1$ at 561 118
\pinlabel $\widehat{U}_2$ at 666 144
\pinlabel $\widehat{E}_2$ at 673 112
\pinlabel $x$ at 680 250
\pinlabel $U'_1$ at 1054 80
\pinlabel $E'_1$ at 1037 52
\pinlabel $U'_2$ at 882 244
\pinlabel $E'_2$ at 916 275
\pinlabel $\widetilde{U}_1$ at 942 170
\pinlabel $\widetilde{E}_1$ at 974 158
\pinlabel $\widetilde{U}_2$ at 945 240
\pinlabel $\widetilde{E}_2$ at 976 212
\pinlabel $\widetilde{U}_3$ at 1014 274
\pinlabel $\widetilde{E}_3$ at 1014 236
\pinlabel $\widetilde{U}_4$ at 1000 145
\pinlabel $\widetilde{E}_4$ at 972 130
\pinlabel $\widetilde{U}_5$ at 1064 149
\pinlabel $\widetilde{E}_5$ at 1089 121
\pinlabel $\widehat{U}'_1$ at 918 147
\pinlabel $\widehat{E}'_1$ at 899 104
\pinlabel $\widehat{U}'_2$ at 1030 139
\pinlabel $\widehat{E}'_2$ at 1004 105
\pinlabel $x'$ at 1095 153
\endlabellist
\centering
\includegraphics[height=4cm]{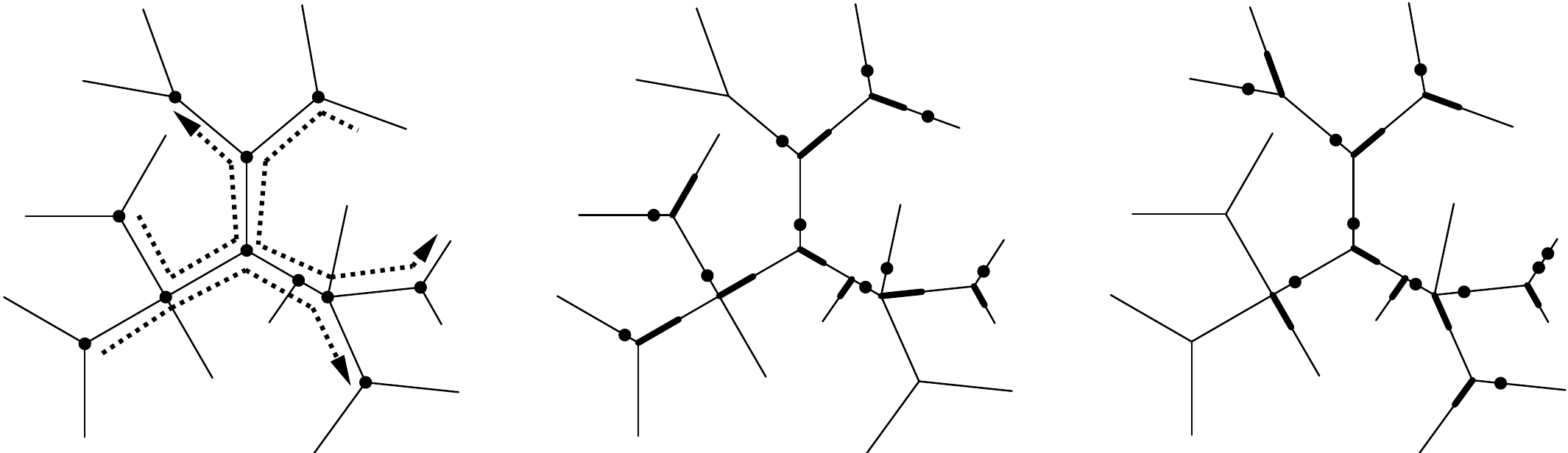}
\caption{The critical cells $\Phi$ and $\Psi$ in $UD^{19}(\Gamma)$ (middle and right)}
\label{fig:PartTwoBii}
\end{figure}
}

\def\RCv{
\begin{figure}[h]
\labellist
\tiny\hair 2pt
\pinlabel $e$ at 182 321
\pinlabel $0$ at 124 284
\pinlabel $1$ at 158 382
\pinlabel $C_\tau$ at 220 400
\pinlabel $3$ at 280 352
\pinlabel $1$ at 258 294
\pinlabel $C_\iota$ at 209 248
\pinlabel $0$ at 410 274
\pinlabel $2$ at 313 153
\pinlabel $0$ at 557 284
\pinlabel $2$ at 600 382
\pinlabel $C_\tau\cup\{\tau(e)\}$ at 653 410
\pinlabel $3$ at 717 352
\pinlabel $1$ at 691 294
\pinlabel $C_\iota\cup\{\iota(e)\}$ at 642 248
\pinlabel $0$ at 843 278
\pinlabel $2$ at 746 153
\endlabellist
\centering
\includegraphics[height=5cm]{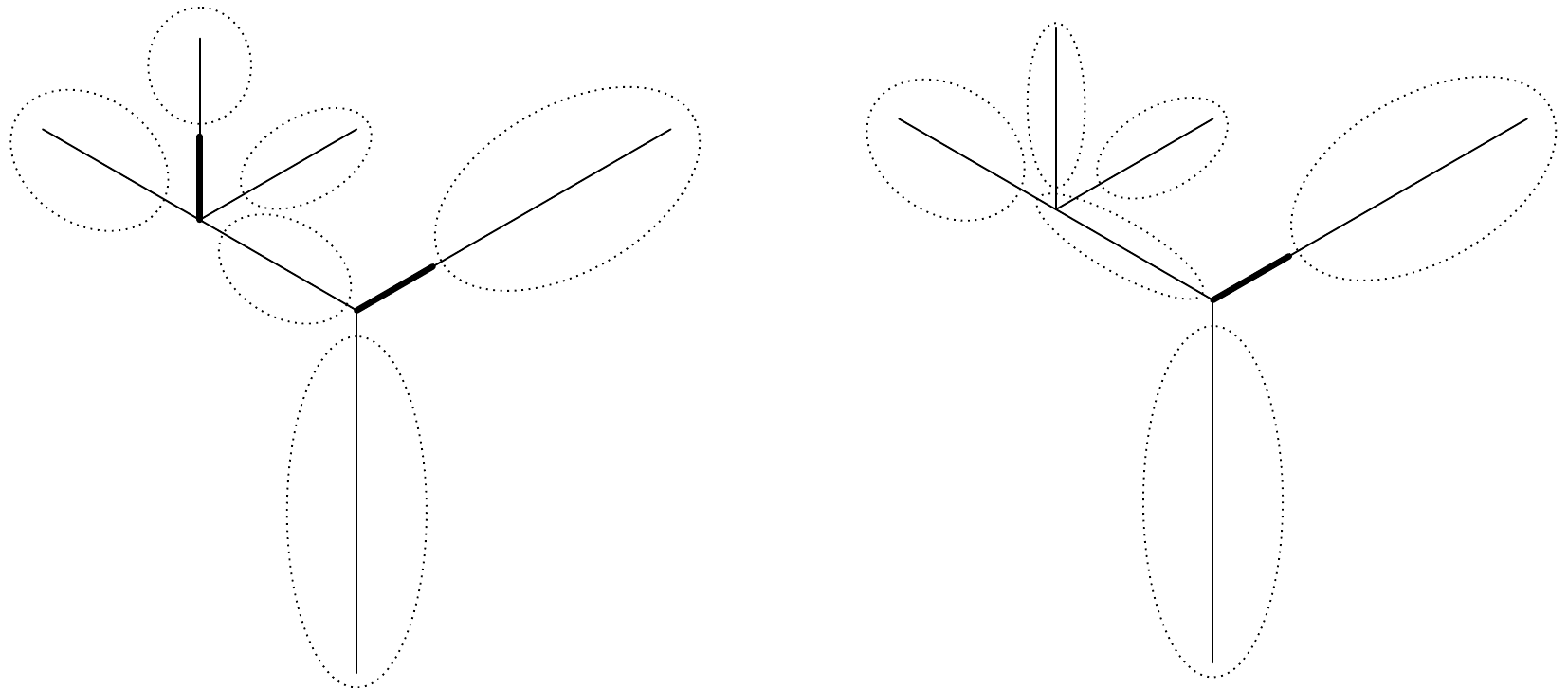}
\caption{The cloud diagram $\mathcal{C}$ for a standard cocycle $\phi_{[c]}$ (left) and the cloud diagram for the cochain $\mathcal{R}_{\mathcal{C},v}$ (right)}
\label{fig:RCv}
\end{figure}
}

\def\ThetaTau{
\begin{figure}[h]
\labellist
\tiny\hair 2pt
\pinlabel $0$ at 104 280
\pinlabel $1$ at 148 332
\pinlabel $3$ at 194 300
\pinlabel $1$ at 196 256
\pinlabel $0$ at 340 266
\pinlabel $2$ at 244 140
\pinlabel $0$ at 439 280
\pinlabel $2$ at 483 324
\pinlabel $2$ at 540 306
\pinlabel $1$ at 531 256
\pinlabel $0$ at 675 266
\pinlabel $2$ at 579 140
\endlabellist
\centering
\includegraphics[height=4cm]{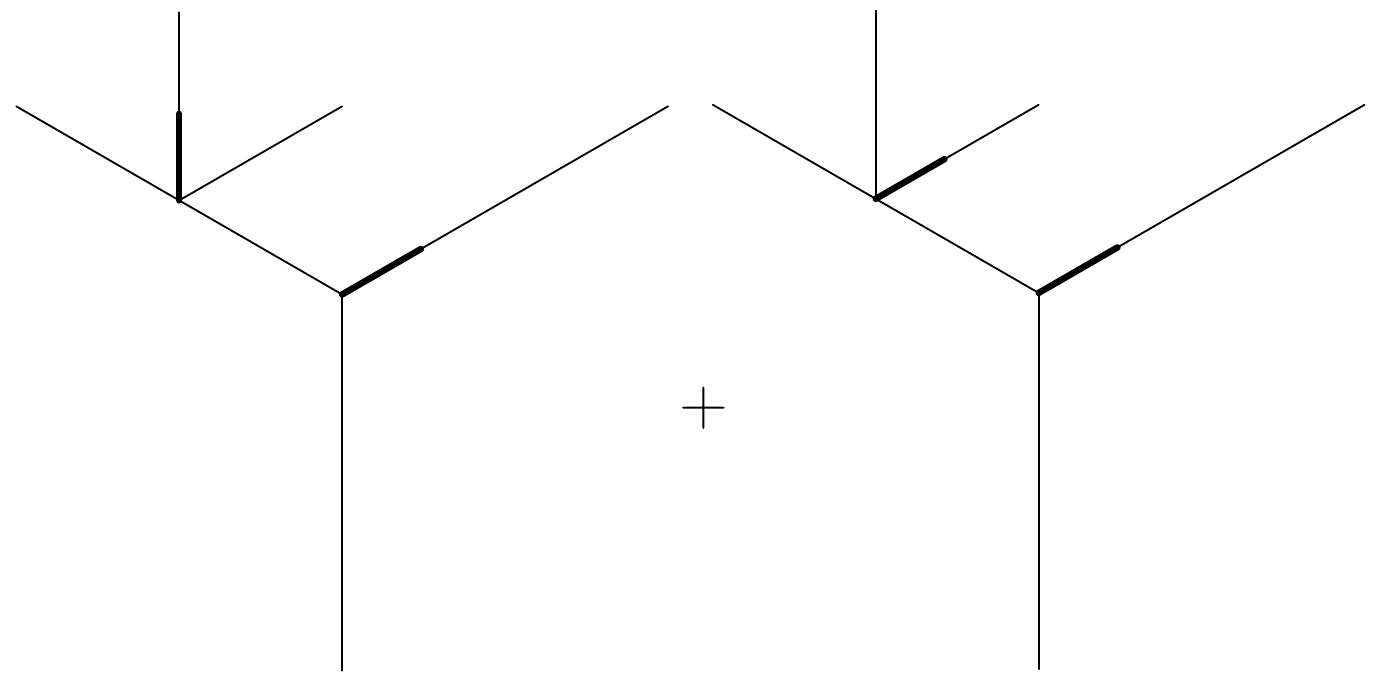}
\caption{The sum $\sum_{i=1}^{d(v)-1}\Theta_{\mathcal{C},v,\tau,i}$ with $\mathcal{C}$ as in Figure \ref{fig:RCv}}
\label{fig:ThetaTau}
\end{figure}
}

\def\ThetaIota{
\begin{figure}[h]
\labellist
\tiny\hair 2pt
\pinlabel $0$ at 100 284
\pinlabel $2$ at 148 332
\pinlabel $3$ at 194 300
\pinlabel $0$ at 196 256
\pinlabel $0$ at 340 266
\pinlabel $2$ at 244 140
\pinlabel $0$ at 449 280
\pinlabel $2$ at 493 324
\pinlabel $3$ at 550 306
\pinlabel $0$ at 541 256
\pinlabel $0$ at 685 266
\pinlabel $2$ at 589 140
\pinlabel $0$ at 802 280
\pinlabel $2$ at 843 324
\pinlabel $3$ at 900 306
\pinlabel $0$ at 896 256
\pinlabel $0$ at 1035 266
\pinlabel $2$ at 941 140
\endlabellist
\centering
\includegraphics[height=4cm]{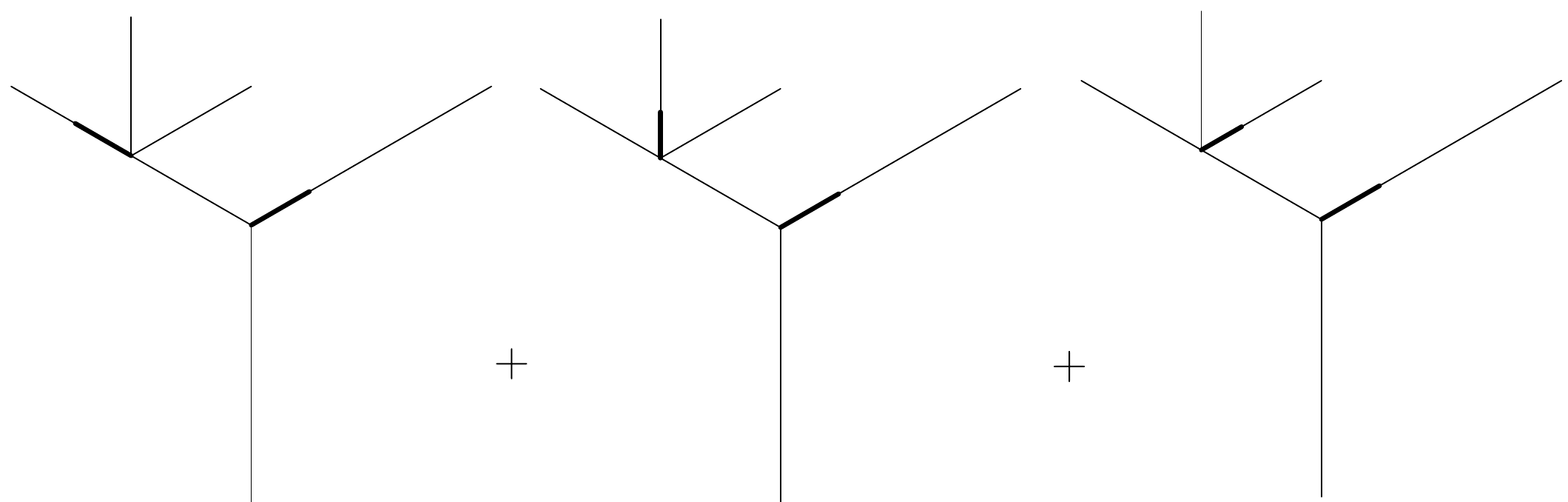}
\caption{The sum $\sum_{i=1}^{d(v)-1}\Theta_{\mathcal{C},v,\iota,i}$ with $\mathcal{C}$ as in Figure \ref{fig:RCv}}
\label{fig:ThetaIota}
\end{figure}
}